\documentclass[12pt]{article}
\usepackage{amsmath, amssymb, amsthm, fullpage}
\usepackage{enumerate}
\usepackage[shortlabels]{enumitem}
\usepackage{float}
\usepackage{tikz-cd}
\setlist[enumerate]{noitemsep}
\usepackage{hyperref}
\usepackage{mathrsfs}
\usepackage{comment}
\usepackage{enumerate}

\usepackage{caption}
\definecolor{pink}{rgb}{1,.2,.6}
\definecolor{blue}{rgb}{.2,.6,.75}
\definecolor{green}{rgb}{.4,.7,.4}

\newtheorem{prop}{Proposition}
\newtheorem{exam}[prop]{Example}

\newtheorem{theorem}[prop]{Theorem}
\newtheorem{cor}[prop]{Corollary}
\newtheorem{lemma}[prop]{Lemma}
\newtheorem{rem}[prop]{Remark}
\newtheorem{conjecture}{Conjecture}

\theoremstyle{definition}

\def\Z{\mathbb{Z}}

\def\Q{\mathbb{Q}}

\def\E{\mathcal{E}}

\title{Constant root number on integer fibres of elliptic surfaces}
\author{Rena Chu$^1$ \and Julie Desjardins$^2$}
\date{%
    $^1$Duke University, Durham, NC, U.S.A.\\%
    rena.chu@duke.edu\\%
    $^2$University of Toronto, Toronto, ON, Canada\\%
    julie.desjardins@utoronto.ca\\[2ex]%
    \today
}

\begin{document}

\maketitle

\begin{abstract}Rizzo showed that the family of elliptic curves $\mathcal{W}(t) :y^2=x^3+tx^2-(t+3)x+1$, a well-known example of Washington, has root number $W(\mathcal{W}(t))=-1$ for all $t\in\Z$ \cite{Riz}. 
In this paper we generalize this example and identify the families of small degree on which this phenomenon happens.  Motivated by results from \cite{BDD} and \cite{Desjardins5}, we study in detail the two families $\mathcal{F}_s(t):y^2=x^3+3tx^2+3sx+st$ and $\mathcal{L}_{w,s,v}(t): wy^2=x^3+3(t^2+v)x^2+3sx+s(t^2+v)$ and describe necessary and sufficient conditions for which subfamilies of $\mathcal{F}_s(t)$ have constant root number on integer fibres. We further prove similar but partial results on $\mathcal{L}_{w,s,v}(t)$. Our results give examples of subfamilies for which there is rank elevation at integer fibres.

\end{abstract}

\section{Introduction}

Throughout this paper\footnote{2010 Mathematics Subject Classification:14J27, 14D10, 11G05\\Keywords: root number; elliptic curves; elliptic surfaces
}, by \emph{family of elliptic curves} -- or simply \emph{family} -- we mean a collection given by the following Weierstrass equation:     \begin{equation}\label{eq: WE}
    \mathcal{E}(t):y^2=x^3+F(t)x+G(t),
    \end{equation}
where $F(t),G(t)$ are elements of $\Q[t]$ such that $\Delta(t)=-16\left(4F(t)^3+27G(t)^2\right)\not=0$.\footnote{We generally exclude the case where $\Delta\in\Q$ as well, since it is the same elliptic curve for each $t$ and is therefore not interesting in our settings.} Each curve of the family obtained by evaluating the Weierstrass equation in \eqref{eq: WE} at $t\in\Q$ is called a \emph{fibre}, as inspired by the theory of elliptic surfaces, and a fibre at $t\in\Z$ is called an \emph{integer fibre}.  Such a family is said to be \emph{isotrivial} if the $j$-invariant $j(t)=-1728(4F(t))^3/\Delta(t)$ is constant, and \emph{non-isotrivial} otherwise. We write $\mathcal{E}$ when we refer to this family as an elliptic surface.

It is widely believed that in a non-isotrivial family of elliptic curves, the rank takes a non-zero value for infinitely many values of $t\in\Q$: this is a variant of the celebrated Goldfeld conjecture, confirmed by the second author \cite{Desjardins1} through the variation of the root number. The root number of an elliptic curve is conjecturally equal to the parity of the geometric rank (weak Birch and Swinnerton-Dyer conjecture), and therefore a negative root number predicts a non-zero rank of the Mordell-Weil group. 

We will say that a family $\E(t)$ has \emph{constant root number over $\Z$} or \emph{over the integer fibres} if the value of $W(\E(t))$ is the same for any choice of $t\in\Z$. It can happen that a family has constant root number on integer fibres. For instance, in 2003 Rizzo showed that the family $$\mathcal{W}(t) :y^2=x^3+tx^2-(t+3)x+1$$ is such that $W(\mathcal{W}(t))=-1$ for every integer $t$ \cite{Riz}.  This family was first studied by Washington \cite{Washington} who proved numerically that the rank of a fibre at $t\in\Z$ with $|t| \leq 1000$ is odd.

\subsection{Results}

In this paper we look at families whose Weierstrass equation has coefficients of bounded degree. Bettin, David and Delauney  \cite{BDD} showed that there are six different non-isotrivial such families and Desjardins \cite{Desjardins5} proved that all but two of these families have varying root numbers on the integer fibres. We are thus motivated to study the two remaining families in more detail:
    \begin{align*}
        \mathcal{F}_s(t)&: y^2=x^3+3tx^2+3sx+st,\\
        \mathcal{L}_{w,s,v}(t)&: wy^2=x^3+3(t^2+v)x^2+3sx+s(t^2+v).
    \end{align*}
We determine with precision the $\mathcal{F}$-families with constant root number on integer fibres, while giving a large class of examples of $\mathcal{L}$-families with this property. (See our upcoming paper for the complete $\mathcal{L}$-family case.) Our main result is the following.

\begin{theorem}\label{mainthm} Let $s, a, b$ be non-zero integers such that $s=-3r^2$ for some non-zero integer $r$. Then the family $\mathcal{F}_s(au+b)$ has constant root number as $u$ varies over $\Z$ if and only if the three conditions holds:
\begin{enumerate}
    \item $\nu_p(b)<\nu_p(a)$ or $\nu_p(s)/2\leq \nu_p(a)\leq \nu_p(b)$ for all $p\geq5$ such that $p\mid s$, and
    \item $\nu_3(b)<\nu_3(a)$ or $(\nu_3(s)-3)/2\leq \nu_3(a)\leq \nu_3(b)$, and 	
    \item one of the following holds:
		\begin{enumerate}
		\item $\nu_2(b)+2<\nu_2(a)$
		\item $\nu_2(b)+2=\nu_2(a)$ and $2\nu_2(b)=\nu_2(s)$
		\item $\nu_2(b)+2=\nu_2(a)$, $2\nu_2(b)\neq \nu_2(s)$ and $\nu_2(s)\equiv 0 \pmod 4$ and
			\begin{enumerate}
			\item $\nu_2(s)-2\nu_2(b)<0$, or
			\item $\nu_2(s)-2\nu_2(b)\equiv 2 \pmod 4, >0$, or
			\item $\nu_2(s)-2\nu_2(b)\equiv 0 \pmod 4, >0$ and $b_2\equiv 1 \pmod 4$
			\end{enumerate}
		\item $\nu_2(b)+2=\nu_2(a)$, $2\nu_2(b)\neq \nu_2(s)$ and $\nu_2(s)\equiv 2 \pmod 4$ and
			\begin{enumerate}
			\item $\nu_2(s)-2\nu_2(b)<0$, or
			\item $\nu_2(s)-2\nu_2(b)=2$ and $b_2\equiv 3 \pmod 4$, or
			\item $\nu_2(s)-2\nu_2(b)\equiv 0 \pmod 4, >0$, or
			\item $\nu_2(s)-2\nu_2(b)=6$, or
			\item $\nu_2(s)-2\nu_2(b)\equiv 2 \pmod 4, >6$ and $b_2\equiv 1 \pmod 4$
			\end{enumerate}
		\item $\nu_2(b)+1=\nu_2(a)$,  $2\nu_2(b)\neq \nu_2(s)$ and 
			\begin{enumerate}
			\item $\nu_2(s)-2\nu_2(b)\leq -4$, or
			\item $\nu_2(s)-2\nu_2(b)=-2$ and $\nu_2(s)\equiv 2 \pmod4$
			\end{enumerate}
		\item $(\nu_2(s)+6)/2\leq \nu_2(a) \leq \nu_2(b)$
		\item $(\nu_2(s)+4)/2 = \nu_2(a) \leq \nu_2(b)$ and $\nu_2(s)\equiv 2 \pmod 4$.
		\end{enumerate}
	\end{enumerate}
\end{theorem}

\begin{exam}\label{ex: basechange}
Observe that $\mathcal{L}_{1,s,v}$ is a subfamily of $\mathcal{F}_s$. Thus it is possible to generate examples of families of the form $\mathcal{L}_{w,s,v}$ with constant root number over $\Z$ by starting from a family $\mathcal{F}_s(au+b)$ with constant root number over $\Z$. For instance the family $$\mathcal{V}_{v}(t):y^2=x^3+(t^2+v)x^2-(t^2+v+3)x+1$$ arises from the Washington family $\mathcal{W}(t)$, which is isomorphic to $\mathcal{F}_{-3^54}(12t+18)$ via a quadratic base change. Therefore every integer fibre of the family $\mathcal{V}_{v}(t)$ also has root number $-1$. Further note that $\mathcal{V}_v(t)$ is isomorphic to $\mathcal{L}_{12,-3^54,v+\frac{3}{2}}(t)$ and in particular that their integer fibres correspond.
\end{exam}

Extending likewise the construction of $\mathcal{L}$-families as quadratic subfamilies of $\mathcal{F}$-families we deduce from Theorem \ref{mainthm} the following result:

\begin{cor}\label{cor: main_to_L} Let $w,r,v\in \Q$ be non-zero and such that $w$, $wv$ and $-3r^2w^2$ are integers. Suppose moreover that those integers respect the three following conditions:
\begin{enumerate}
    \item $\nu_p(v)<0$ or $\nu_p(r)\leq0\leq\nu_p(v)$, for every $p\mid rw$,
    \item $\nu_3(v)<0$ or $\nu_3(r)-1\leq 0\leq \nu_3(v)$,
    \item one of the following holds:
    \begin{enumerate}
        \item $\nu_2(v)\leq-2$,
        \item $\nu_2(v)=-1$ and 
        \begin{enumerate}
            \item $\nu_2(r)\leq-2$, or
            \item $\nu_2(r)=-1$ and $\nu_2(w)$ is even
        \end{enumerate}
        \item $\nu_2(r)+3\leq0\leq\nu_2(v)$,
        \item $\nu_2(r)+2=0\leq \nu_2(v)$.
    \end{enumerate}
\end{enumerate}
Then for any $a,b\in\Z$, the family $\mathcal{L}_{w,-3r^2,v}(au+b)$ has constant root number as $u$ varies over $\Z$.
\end{cor}

\begin{rem} Theorem \ref{mainthm} does not describe every case of families of the form $\mathcal{L}_{w,s,v}(au+b)$ with constant root number over $\Z$. There are many more such families, in particular coming as quadratic subfamilies of $\mathcal{F}$-families with non-constant root number, as illustrated by the following example.
\end{rem}

\begin{exam}\label{ex: notcomingfromconstant}
The family $\mathcal{L}_{7,-3\cdot2^27^2,1}(4u+2)$ has constant root number for $u\in\Z$ by Lemma \ref{mainthmL} and is a quadratic subfamily (when $t=(4u+2)^2$) of $\mathcal{F}_{-3\cdot2^27^4}(7t+7)$, for which the root number is non-constant on integer fibres.
\end{exam}
\begin{rem}
While writing this paper, we found an impressive number of technical conditions for the general case of $\mathcal{L}$-families, the length of which convinced us to leave a more detailed inspection of this latter case to an upcoming second paper. 
\end{rem}

\subsection{Consequences}

Theorem \ref{mainthm}, coupled with two analytic number theory conjectures on the primitive factors of the discriminant of the families, leads to a description of families of elliptic curves of coefficients of bounded degree that have constant root number over $\Z$. This description is stated in the following:

\begin{theorem}\label{cor: main}
Let $\E(t)$ be a family of elliptic curves for $t\in\Q$ given by the Weierstrass equation 
\begin{equation}\label{oursurfaces}\E(t):y^2=x^3+a_2(t)x^2+a_4(t)x+a_6(t)\end{equation}
where $\deg a_i\leq 2$ for $i=2,4,6$. Define
    \begin{align}\label{def:MandB}
        M_\E:=\prod_{P\mid \Delta, P\nmid c_4}{P}, \qquad B_\E:=\prod_{P\mid \Delta}{P}
    \end{align}
    to be the product of places of multiplicative reduction and the product of places of bad reduction, respectively.
Suppose Chowla's conjecture and the Squarefree conjecture holds for $M_\E$ and $B_\E$, respectively. Then
    \begin{align*}
        W_\pm(\E(t),\mathbb{Z}):=\#\{t\in\Z \ \vert \ W(\E(t))=\pm1\}
    \end{align*}
    are both infinite except possibly if $\E(t)$ has the form $\mathcal{F}_s(au+b)$ or $\mathcal{L}_{w,s,v}(au+b)$ with $s=-3r^2$ for some non-zero $r\in \Z$. Moreover, if $\E(t)$ has the form of a $\mathcal{F}$-family then $W_\pm(\E(t),\mathbb{Z})$ are both infinite except if and only if it respects the conditions of Theorem \ref{mainthm}.
\end{theorem}

This is a direct consequence of \cite[Theorem 1.1]{Desjardins5}. This result is conditional on two analytic number theory conjectures that are true for low-degree polynomials. The first is Chowla's conjecture which predicts the behavior of $\lambda(n)$, the parity of the number of prime factors of $n$, when $n$ varies through the values of a polynomial: $$\sum_{\mid t\mid\leq X}{\lambda(f(t))}=o(X).$$
The conjecture is known to be true for $\deg f\leq1$, which makes Theorem \ref{cor: main} unconditional relatively to this conjecture for families with at most one (linear) factor of the discriminant $\Delta$ not dividing the $c_4$-invariant.

The second conjecture is the Squarefree conjecture which estimates the proportion of squarefree values of a polynomial. For simplicity, suppose that $\gcd\{f(x)\vert x \in\Z\}$ is squarefree. According to this conjecture we have:
$$Sqf(X)=\frac{1}{N}{\prod_{p\leq N}\left(1-\frac{t_f(p)}{p^{2}}\right)}X+o(X),$$
where $t_f(p)$ denotes the number of solutions modulo $p^2$ of $f(x)\equiv0\pmod{p^2}$.
This conjecture is known to hold when every factor of $f$ has degree at most 3. Thus the only case where Theorem \ref{cor: main} is conditional relative to this conjecture is the case where the $c_4$-invariant of a family is irreducible over $\Q(t)$.\\ 

Another consequence of Theorem \ref{mainthm} is the ability to predict a \textit{rank elevation} of the integer fibres compared to the \textit{generic rank} $r(\E)$ of a family with constant root number equal to $(-1)^{r(\E)}$. Here we use Silverman's Specialization Theorem (Theorem \ref{thm: silv}) and the Parity Conjecture (Conjecture \ref{parityconjecture}). The particular case that draws our attention in this paper is that on families $\mathcal{F}_{s}$ where $s=-12k^4$ for some $k$. It was proven by \cite[Proposition 5]{BDD} that such families have generic rank $1$ (and $0$ if $s$ is of another form). Using \eqref{eq: 9} and Propositions \ref{prop: s=-12k^4 p5}, \ref{prop: s=-12k^4 p3} and \ref{prop: s=-12k^4 p2}, we 
list every family of the form $\mathcal{F}_{-12k^4}(au+b)$ with constant root number equal to $1$.

\subsection{Families with coefficients of bounded degrees}\label{sectionexample}
We say a family of elliptic curves $\E(t)$ has coefficients of bounded degree if it is given for $t\in\Q$ by the Weierstrass equation 
\begin{equation}\label{oursurfaces}\E(t):y^2=x^3+a_2(t)x^2+a_4(t)x+a_6(t)\end{equation}
where $\deg a_i\leq 2$ for $i=2,4,6$. Suppose there is no factor of $\Delta$ that divides the $c_4$-invariant (or in other words there is no place of multiplicative reduction i.e. the family is \emph{potentially parity-biased}). Then Bettin, David and Delauney \cite[Theorem 7 and 8]{BDD} proved that there are essentially six different classes of such families that are non-isotrivial. We list these families here along with their primitive factors and corresponding Kodaira types:
	\begin{table}[H]
	\small
	\caption{Families with coefficients of bounded degree}
	\label{table: kodairaBDD}
	\noindent \makebox[\textwidth]{
	\begin{tabular}{l l l l}
	\hline
	$\mathcal{F}_s(t): y^2=x^3+3tx^2+3sx+st$	&$t^2-s$ (II)	\\
	$\mathcal{G}_w(t):wy^2=x^3+3tx^2+3tx+t^2$	&$t-1$(II), $t$(III)\\
	$\mathcal{H}_w(t):wy^2=x^3+(8t^2-7t+3)x^2+3(2t-1)x+(t+1)$	&$t$ (III), $t^2-\frac{11}{8}t+1$ (II)\\
	$\mathcal{I}_w(t): wy^2=x^3+t(t-7)x^2-6t(t-6)x+2t(5t-27)$	&$t$ (II), $t^2-10t+27$ (II)\\
	$\mathcal{J}_{m,w}(t): wy^2=x^3+3t^2x^2-3mtx+m^2$	&$t^3+m$ (II)\\
	$\mathcal{L}_{w,s,v}(t): wy^2=x^3+3(t^2+v)x^2+3sx+s(t^2+v)$	&$t^4+2vt^2+v^2-s$ (II)\\
	\hline
	\end{tabular}
	}
	\end{table}

In \cite{Desjardins5}, the second author proved (in particular) that any family isomorphic to one of the form $\mathcal{G}_w$, $\mathcal{H}_w$, $\mathcal{I}_w$ and $\mathcal{J}_{m,w}$ has infinitely many integer fibres with negative (resp. positive) root number. Her result is unconditional in those cases, since there is no multiplicative reduction (hence no need of Chowla's Conjecture) and the primitive factors of $\Delta$ have degree 3 or less (hence the Squarefree Conjecture is verified). 
Hence the only families among the six classes on which \cite{Desjardins5} leaves an uncertainty are those of the form $\mathcal{F}_s$ and $\mathcal{L}_{w,s,v}$ in very special circumstances. As mentioned earlier we note that $\mathcal{L}$-families are quadratic subfamilies of $\mathcal{F}$-families.

Chinis \cite{Chi} computed the average root number on the families $\mathcal{F}_s$ and showed that $\mathcal{F}_s(t)$ is \emph{parity biased} over $\Z$ (i.e. the average root number over $\Z$ is not 0) if and only if $s\not\equiv1,3,5\pmod 8$.
The \emph{average root number} of a family of elliptic curves $\E$ over $\Z$ is defined as $$Av_\Z(W(\E)):=\lim_{T\rightarrow\infty}{\frac{1}{2T}\sum_{\vert t \vert\leq T}{W(\E_t)}},$$ provided that the limit exists.

During our initial investigation, we observed that the families $\mathcal{F}_s(au+b)$ with constant root number when $u$ varies through $\Z$ are exactly those with $Av_\mathbb{Z}\left(W(\mathcal{F}_s(au+b)\right)=\pm1$. This remark is not proven in the present paper, but is included in the first author's URSA report \cite{reportRena}. 

\subsection{Outline of the paper}

In Section 2 we recall on the root number of an elliptic curve, with special attention to a decomposition of the root number into local functions $w_p^*$ when studied in families. Moreover we prove Theorem \ref{cor: main}.

Section 3 contains the proof of our main results. We begin with a simpler result in Section \ref{sec: 3.1} on the root number of integer fibres of the family $\mathcal{F}_s(t)$ as $t$ runs through $\Z$. This motivates us to study subfamilies $\mathcal{F}_s(au+b)$ as $u$ runs through $\mathbb{Z}$ in Section \ref{sec: 3.2}, where we give necessary and sufficient conditions on $s, a, b$ for the root number to be constant. This is our main result which is stated in Theorem \ref{mainthm}. In Section \ref{sec: 3.3} we prove similar results concerning the families $\mathcal{L}_{w,s,v}(au+b)$.

Finally in Section \ref{sec: 4}, we look for rank elevation in the integer fibres of $\mathcal{F}$-families: i.e. if we denote by $r(\mathcal{F}_s)$ the generic rank of the family, when does $r(\mathcal{F}_s(au+b))>r_{\mathcal{F}_s}$ for every $u\in\Z$ (except possibly a finite number of $u$)? We consider families $\mathcal{F}_s(au+b)$ with $s,a,b\in\Z$ and such that $s=-12k^4$, for some $k\in\Z$. In this particular case, the generic rank is 0, as proved in \cite{BDD}, and 1 if $s$ has another form. Based on results from the previous section, we describe which families have constant root number $1$ over integer fibres and give some examples. It is predicted by the parity conjecture (weak BSD) that the root number predicts the parity of the rank: $W(E)=(-1)^{r(E)}$. Thus there should be a rank elevation when the generic rank is even (resp. odd) and the root number is $1$ (resp. $-1$). 

The Appendix contains our formulae for the function denoted $w_p^*$, equal up to some factor to the local root number at a prime $p$. This function is defined in \eqref{def:local_functions} of Section \ref{sec:define_mod_root_number} and appears continuously throughout this paper.

\subsection{Acknowledgments}

This paper was initiated as part of a summer NSERC's USRA research project of the first author supervized by the second author and Florian Herzig. Both authors are grateful to him for co-supervizing, for his financial support and for his helpful remarks. The exposition of the paper was improved by suggestions of Marc Hindry, Matthew Bisatt and Paul Voutier. The authors thank the University of Toronto for making the project possible. 
\section{Root number}

\subsection{Definitions}\label{sec:def_root}

Let $E$ be an elliptic curve over $\Q$. The \emph{root number} $W(E)$ of $E$ can be defined in two equivalent ways (since the base field is the rationals). The first is as the sign of the functional equation of the $L$-function attached to $E$; it is equal to the parity of the analytic rank $r_{an}$ (the order of vanishing of $L(E,s)$ at $s=1$):
$$W(E)=(-1)^{r_{an}(E)}.$$
This is made possible by the Modularity Theorem by Wiles \cite{Wil95} and Breuil-Conrad-Diamond-Taylor \cite{BCDT01} which guarantees the existence of the functional equation. 
One of the consequences of the famous Birch and Swinnerton-Dyer conjecture is that the analytic and the geometric rank of an elliptic curve have the same parity. Therefore the following conjecture can be seen as a weak version:

\begin{conjecture}[Parity conjecture]\label{parityconjecture}$W(E)=(-1)^{r(E)}.$
\end{conjecture} 
As a consequence of this conjecture, $W(E)=-1$ implies that $E$ has a non-zero rank. This first definition of the root number is difficult to compute, so instead we make use of the following equivalent definition in which the root number is a product of local factors
    \begin{align}\label{eqn:local_factors}
        W(E)=\prod_{p\leq \infty}{W_p(E)}
    \end{align}
where each factor $W_p(E)$ is called the \emph{local root number at $p$}, defined in terms of the epsilon factors of the Weil-Deligne representations of $\Q_p$. This decomposition makes computation easier by the formulas and tables found in papers of Rohrlich \cite{Rohrlich} for $p\geq5$ and Rizzo \cite{Riz} for $p=2,3$.

\subsection{Insipid places and proving Theorem \ref{cor: main}}

Recall that to each finite place of $\Q(t)$ is attached an irreducible polynomial $P$. We say a family has an \emph{insipid place at $P$} if the place is finite of type: I$_0^*$; II, II$^*$, IV or IV$^*$ with $P=A^2+3B^2$; or III, III$^*$ with $P=A^2+B^2$, for some non-zero $A,B\in\Q[t]$.
We will see how Theorem \ref{mainthm} and the following result of the second author implies Theorem \ref{cor: main}.

\begin{theorem}\cite[Theorem 1.1]{Desjardins5}
Let $\E(t)$ be a family of elliptic curves over $\Q$. Suppose there exists at least one irreducible primitive polynomial $P_0(T)\in\Z[T]$ associated to either a place that is not insipid for $\E$ or to a place of type  $I_0^*$ and $\deg P_0$ is odd.
Suppose moreover that $M_\E$ respects Chowla's conjecture or $M_\E=1$, and $B_\E$ respects the Squarefree conjecture, where $M_\E$ and $B_\E$ are as defined in \eqref{def:MandB}.
Then the sets $W_+(\E,\Z)$ and $W_-(\E,\Z)$ are both infinite.
\end{theorem}
Let $\E(t)$ be a family of elliptic curves with coefficients $a_i(t)\leq2$ for all $i=2,4,6$.
Suppose that $\E(t)$ has only insipid places of bad reduction, except perhaps at $-\mathrm{deg}$. Then by \cite{BDD} $\E(t)$ is isomorphic either to $\mathcal{F}_s(t)$ or $\mathcal{L}_{w,s,v}(t)$. Indeed, as is given in Table \ref{table: kodairaBDD}, a family $\mathcal{G}_w(t)$ has a linear place of type II and another of type III: neither $t$ or $t-1$ can be written as $t=A^2+B^2$ for non-zero $A,B\in\Q[t]$, and we can similarily exclude the families $\mathcal{H}_w(t)$, $\mathcal{I}_w(t)$ and $\mathcal{J}_{m,w}(t)$.

Moreover, observe that a $\mathcal{F}$-family has an insipid place at $t^2-s$ if and only if $s=-3r^2$ for some $r\in\Z$. This proves the first statement of Theorem \ref{cor: main}.

To prove the second statement of Theorem \ref{cor: main}, observe that we are left to studying the root number of the fibres of the families $\mathcal{F}_{-3r^2}(au+b)$ and that it is exactly what Theorem \ref{mainthm} does.

\subsection{Decomposition of the root number into product of the local contributions}\label{sec:define_mod_root_number}
In order to study the variation of the root number $W(\E_t)$ of a family of elliptic curves $\E_t$ more easily, we reduce this to the study of local root numbers at each prime $p$. As remarked in \S \ref{sec:def_root}, this makes the computation of the root number more accessible. However the decomposition in the form of \eqref{eqn:local_factors} does not yet allow us to study the local factors independently. Thus we formulate another decomposition of $W(\E_t)$, already used in \cite{Manduchi}, \cite{Helfgott} and \cite{Desjardins1}. 

For each pair of integers $(a,b)\in\Z\times\Z$ and even integer $\delta$, define the \textit{modified Jacobi symbol} by
    \begin{equation}\label{quadraticsymboldef}\left(\frac{a}{b}\right)_\delta:=\prod_{p\nmid\delta}{\left(\frac{a_{p}}{p}\right)^{\nu_p(b)}}\end{equation}
where the product runs through the prime number $p\nmid\delta$, $(\frac{\cdot}{p})$ is Legendre symbol, and $a_{p}$ is the integer such that $a=p^{\nu_p(a)}a_{p}$. If $a,b,\delta$ are two-by-two coprime, then the symbol $(\frac{a}{b})_\delta$ is the classical Jacobi symbol.

Let $\E$ be a family of elliptic curves, and denote by $\E(t)$ each of the fibres at $t\in\Z$. Suppose that $\E$ has only places of reduction II, II$^*$ or I$_m^*$, except perhaps at $-\mathrm{deg}$. (Then, as we saw in the previous subsection, if we suppose that the coefficients $a_i(t)\leq2$ for all $i=2,4,6$, then by \cite{BDD} $\E$ is isomorphic either to $\mathcal{F}_s(t)$ or $\mathcal{L}_{w,s,v}(t)$.) Let $B(t)$ be the product of all finite places of $\Q(t)$ such that $\E$ has bad reduction. Then for every prime $p$, define the local functions $w_p^*(\E(t)):\mathbb{Z}_p\rightarrow\mathbb{C}$ by
    \begin{align}\label{def:local_functions}
        w_p^*(\E(t)):=\begin{cases}
        \mathrm{sgn}(B(t))\left(\frac{-1}{B(t)}\right)_{2}W_2(\E(t)),& \text{if }p=2\\
        (-1)^{\nu_3(B(t))}W_3(\E(t)),& \text{if }p=3\\
        W_p(\E(t))\left(\frac{-1}{p}\right)^{\nu_p(B(t))},& \text{if }p\geq 5.
\end{cases}
    \end{align}
When the family of elliptic curves is clear from context, we write $w_p^*(t)$ in place of $w_p^*(\E(t))$.
Propositions \ref{prop: w_p^*(t)}, \ref{prop: w_3^*(t)} and \ref{prop: w_2^*(t)} give the value of these local functions when $\E(t)=\mathcal{F}_s(t)$ according to the parameter $s$. Having defined the local function $w_p^*$, we now have the following decomposition of the root number.

\begin{theorem}\label{thm: indep} Let $\E(t)$ be a family of elliptic curves. Then $$W(\E(t))=-\prod_{p<\infty}w_p^*(\E(t)).$$
Moreover each of the functions $w_p^*(\E(t))$ is locally constant outside a finite set of points. 
\end{theorem}

This guarantees a certain independency of those functions, so that the global root number varies over $\Z$ if and only if one of them varies over $\Z_p$ for some $p$.

\begin{proof}
These functions correspond to those that we find using \cite[Proposition A.4]{Desjardins1}: they are locally constant because they are finite products of Jacobi or Hilbert symbols. The equality follows directly from \cite[Theorem 3.4]{Desjardins1} since we have 
$$W(\E(t))=-\left(\frac{-1}{B(t)}\right)_\delta\prod_{p\mid\delta}W_p(\E(t)).$$
Indeed, since every factor $P(t)$ of $B(t)$ is insipid, the $h$-contribution defined in \cite{Desjardins1} -- usually more complicated -- is simply $h_P(t)=+1$ in our case. To complete the proof, observe that $$\left(\frac{-1}{B(t)}\right)_\delta=\mathrm{sgn}(B(t))\left(\frac{-1}{B(t)_{2}}\right)(-1)^{\nu_3(B(t))}\prod_{p<\infty,\not=2,3}\left(\frac{-1}{p}\right)^{\nu_p(B(t))}.$$
\end{proof}


\subsection{Decomposition of the root number for $\mathcal{F}_s(t)$ and $\mathcal{L}_{w,s,v}(t)$}\label{section:familyL}
Let $s\in\Z$ and consider a family of the form
$\mathcal{F}_s(t): y^2=x^3+3tx^2+3sx+st$. We have the following $c_4$- and $c_6$-invariants and the discriminant:
    \begin{align*}
        c_4(t)&=2^43^2(t^2-s),\\
        c_6(t)&=-2^63^3st(t^2-s),\\ \Delta(t)&=-2^63^3s(t^2-s)^2.
    \end{align*}
The only finite place of bad reduction is $P(t)=t^2-s$, which has Kodaira type II. Moreover, a quick observation gives $\delta=6$ and $M_\E(t)=1$. By \cite[Theorem 3.4]{Desjardins5}, we have
    \begin{align}\label{eqn:decomp_root}
        W(\mathcal{F}_s(t))&=-W_2(\mathcal{F}_s(t))W_3(\mathcal{F}_s(t))\left(\frac{-1}{P(t)_{(2)}}\right)\\
        &\qquad \times \prod_{p\geq5,p^2\mid t^2-s}{\begin{cases}\left(\frac{-3}{p}\right),&\nu_p(t^2-s)\equiv2,4\mod 6\\1,&\text{otherwise.}\end{cases}}\nonumber
    \end{align}
From the definition of $w_p^*(t)$ from \eqref{def:local_functions}, we get
    \begin{equation}\label{eq: w_p^*}
    w_p^*(\mathcal{F}_s(t))=
    \begin{cases}
    \mathrm{sgn}(t^2-s)\left(\frac{-1}{t^2-s}\right)_{2}W_2(\mathcal{F}_s(t)),& \text{if }p=2\\
    (-1)^{\nu_3(t^2-s)}W_3(\mathcal{F}_s(t)),& \text{if }p=3\\
    \left(\frac{-1}{t^2-s}\right)_{p}
    \times \begin{cases}\left(\frac{-3}{p}\right),&\nu_p(t^2-s)\equiv2,4\pmod 6\\1,&\text{otherwise}
    \end{cases},& \text{if }p\geq 5.
\end{cases}\end{equation}
and hence we have the following consequence of Theorem \ref{thm: indep}:
    \begin{lemma} $W(\mathcal{F}_s(t))=-\prod_{p<\infty}w_p^*(\mathcal{F}_s(t))$.
        \end{lemma}
We note that each $w_p^*(\mathcal{F}_s(t))$ is independent from the others and so the root number is constant if and only if $w_p^*(\mathcal{F}_s(t))$ is constant for each $p$.
We use formulas from \cite{Chi} and \cite{BDD}, slightly modified to satisfy our needs, to compute the functions $w_p^*$. These are reported in Appendix \ref{sec: app} in Proposition \ref{prop: w_p^*(t)} for $p\geq5$, Proposition \ref{prop: w_3^*(t)} for $p=3$ and Proposition \ref{prop: w_2^*(t)} for $p=2$. 

    Now we write a similar decomposition for the family $\mathcal{L}_{w,s,v}(t)$. Let $T=t^2+v$, so that
    \begin{align*}
        c_4(\mathcal{L}_{w,s,v}(t))&=2^43^2((Tw)^2-sw^2),\\ c_6(\mathcal{L}_{w,s,v}(t))&=-2^63^3Tw((Tw)^2-sw^2),\\
        \Delta(\mathcal{L}_{w,s,v}(t))&=-2^63^3sw^2((Tw)^2-sw^2)^2.
    \end{align*}
This means that $\mathcal{L}_{w,s,v}(t)=\mathcal{F}_{sw^2}(w(t^2+v))$ and so using \eqref{eqn:decomp_root} we get
    \begin{align*}
        W(\mathcal{L}_{w,s,v}(t))
        &=W\left({\mathcal{F}_{sw^2}}\left(w\left(t^2+v\right)\right)\right)\\
        =&-W_2\left(\mathcal{F}_{sw^2}\left(w(t^2+v)\right)\right)W_3\left(\mathcal{F}_{sw^2}\left(w(t^2+v)\right)\right)\left(\frac{-1}{P(T)_{(2)}}\right)\\
	&\times \prod_{p\geq5,p^2\mid T^2-s}{\begin{cases}\left(\frac{-3}{p}\right),&\nu_p(T^2-s)\equiv2,4\pmod 6\\1,&\text{otherwise}\end{cases}}
    \end{align*}
from which we may deduce the following.
\begin{lemma}\label{lem: familyL}
$W\left(\mathcal{L}_{w,s,v}(t)\right)=-\prod_{p<\infty}{w_p^*\left(\mathcal{F}_{sw^2}(w(t^2+v))\right)}.$
\end{lemma}

\section{Proof of Theorem \ref{mainthm}}

To emphasize the need to consider families given by specializing $t$ to arithmetic progressions $a\Z+b \subset \Z$, we insist on the fact that we look at the \textit{integer fibres} of families of elliptic curves. This means that we need to distinguish between families who are isomorphic over $\Q$.  In Section \ref{sec: 3.1}  we illustrate this fact using the result of Theorem \ref{thm: simpler}. It says that no family of the form $\mathcal{F}_s(t)$ has constant root number over $\Z$; however, we know that Washington's example does satisfy this property. This is because Washington's example is a linear subfamily of a $\mathcal{F}_s$ family.

Thus we must extend our study of the function $w_p^*$ to the subset $a\mathbb{Z}+b\subset\mathbb{Z}$; in other words, study the variation of the root number on the surfaces $\mathcal{F}_s(au+b)$ and $\mathcal{L}_{w,s,v}(au+b)$, families parametrized by the variable $u \in \Z$. Although these are technically isomorphic over $\Q$ to respectively $\mathcal{F}_s(t)$ and $\mathcal{L}_{w,s,v}(t)$, their integer fibres do not correspond.

Therefore we need in order to be general to include in our study all linear subfamilies of $\mathcal{F}_s(t)$, that is to say those where we take $t\in a\mathbb{Z}+b$ (write $t=au+b$ and take $u\in\Z$). This is what we do in Section \ref{sec: 3.2}.

In Sections \ref{sec: 3.1} and \ref{sec: 3.2} we write $w_p^*(t)$ in place of $w_p^*(\mathcal{F}_s(t))$ and in Section \ref{sec: 3.3} $w_p^*(t)$ replaces $w_p^*(\mathcal{L}_{w,s,v}(t))$.

\subsection{Families $\mathcal{F}_s(t)$ with $t\in \Z$}
\label{sec: 3.1}

We begin with this special case of our main theorem where we run the parameter $t$ through all integers. 

\begin{theorem} \label{thm: simpler}For every fixed $s\in \mathbb{Z}$, $W_{\pm}(\mathcal{F}_s(t),\mathbb{Z}) = \infty$.
	In other words, there does not exist $s\in \mathbb{Z}$ such that $\mathcal{F}_s(t)$ has constant root number for all $t\in \mathbb{Z}$.
\end{theorem}
	\begin{proof}
	By Theorem \ref{thm: indep}, we have $W(\mathcal{F}_s(t))=-\prod_{p<\infty}{w_p^*(t)}$, where the functions $w_p^*(t)$ are independent from one another. This means that if there exists a prime $p$ such that $w_{p}^*(t)$ is not constant for all $t\in \Z$, then the root number $W\left(\mathcal{F}_s(t)\right)$ is not constant. The result then follows from Lemma \ref{constant local}.
	\end{proof}

\begin{lemma}\label{constant local}	
	Fix an $s \in \Z$. Then $w_2^*(t)$ is not constant for $t\in \Z$.
\end{lemma}
	\begin{proof}

		 We show that $w_2^*(t)$ is not constant by finding integers $t$ and $t'$ with different local root numbers at $p=2$.
	    We analyze case by case 
		using Proposition \ref{prop: w_2^*(t)}.
				
			\begin{enumerate}[wide=0pt]
			\item    If $\nu_2(s)=0$, then $\nu_2(s)-2\nu_2(t)\leq 0$. We look at cases $s_2 \pmod{16}$.
				\begin{itemize}
				\item    If $s_2\equiv 3\pmod{4}$, choose $t, t'$ s.t. $\nu_2(t)=1, \nu_2(t')=2$. Then $\nu_2(s)-2\nu_2(t)=-2$ so $w_2^*(t)=1$, while $\nu_2(s)-2\nu_2(t')=-4$ so $w_2^*(t')\equiv s_2\pmod{4}=-1$. 
				\item    If $s_2\equiv 1, 13\pmod{16}$, choose $t, t'$ such that (i) $\nu_2(t)=\nu_2(t')=1$ and (ii) $t_2\equiv 1 \pmod{4}, t'_2\equiv 3 \pmod{4}$. Then $\nu_2(s)-2\nu_2(t)=\nu_2(t)-2\nu_2(t')=-2$ and so $w_2^*(t)=1$ and $w_2^*(t')=-1$. 
				\item    If $s_2\equiv 5, 9\pmod{16}$, choose $t, t'$ such that (i) $\nu_2(t)=\nu_2(t')=1$ and (ii) $t_2\equiv 1 \pmod{4}, t'_2\equiv 3 \pmod{4}$. Then $\nu_2(s)-2\nu_2(t)=\nu_2(t)-2\nu_2(t')=-2$ and so $w_2^*(t)=-1$ and $w_2^*(t')=1$. 
				\end{itemize}
			\item    If $\nu_2(s)\equiv 0 \pmod{4}, >0$, let $\ell=(\nu_2(s)-2)/2$. Choose $t, t'$ such that (i) $\nu_2(t)=\nu_2(t')=\ell$ and (ii) $t_2\equiv s_2\pmod{4}, t'_2\equiv -s_2\pmod{4}$. Then $\nu_2(s)-2\nu_2(t)=\nu_2(s)-2\nu_2(t')=2$ and so $w_2^*(t)=1$ and $w_2^*(t')=-1$.\\
			
			\item    If $\nu_2(s)\equiv 1\pmod{4}$, let $\ell=(\nu_2(s)-1)/2$. Choose $t, t'$ such that $\nu_2(t)=\nu_2(t')=\ell$. Then $\nu_2(s)-2\nu_2(t)=\nu_2(s)-2\nu_2(t')=1$. Note that $t^2-s=2^{2l}t^2_2-2^{\nu_2(s)}s_2=2^{2l}(t^2_2-2s_2)$ where $2\nmid t_2^2-2s_2$, so $(t^2-s)_2=t_2^2-2s_2\equiv 1-2s_2\pmod{4} \equiv 3 \pmod{4}$. Then
				\begin{itemize}
				\item    If $s_2\equiv 1\pmod{4}$, choose $t, t'$ such that (i) $\nu_2(t)=\nu_2(t')=\ell$ and (ii) $t_2\equiv 1 \pmod{8}, t'_2\equiv 3\pmod{8}$. Then $w_2^*(t)\equiv (t^2-s)_2 \pmod{4} =-1$ and $w_2^*(t')\equiv -(t^2-s)_2\pmod{4} =1$. 
				\item    If $s_2\equiv 3\pmod{4}$, choose $t, t'$ such that (i) $\nu_2(t)=\nu_2(t')=\ell$ and (ii) $t_2\equiv 1 \pmod{8}, t'_2\equiv 5\pmod{8}$. Then $w_2^*(t)\equiv (t^2-s)_2 \pmod{4} =-1$ and $w_2^*(t')\equiv -(t^2-s)_2\pmod{4} =1$. 
				\end{itemize}
				
			\item    If $\nu_2(s)=2$, then $\nu_2(s)-2\nu_2(t)\leq 2$. We look at cases $s_2 \pmod{16}$.
				\begin{itemize}
				\item    If $s_2\equiv 1\pmod{4}$, choose $t, t'$ such that $\nu_2(t)=2, \nu_2(t')=3$. Then $\nu_2(s)-2\nu_2(t)=-2$ so $w_2^*(t)=-1$. And $\nu_2(s)-2\nu_2(t')=-4$ so $w_2^*(t')\equiv s_2\pmod{4}=1$. 
				\item    If $s_2\equiv 3, 7\pmod{16}$, choose $t, t'$ such that (i) $\nu_2(t)=\nu_2(t')=2$ and (ii) $t_2\equiv 1 \pmod{4}, t'_2\equiv 3 \pmod{4}$. Then $\nu_2(s)-2\nu_2(t)=\nu_2(t)-2\nu_2(t')=-2$ and so $w_2^*(t)=1$ and $w_2^*(t')=-1$. 
				\item    If $s_2\equiv 11, 15\pmod{16}$, choose $t, t'$ such that (i) $\nu_2(t)=\nu_2(t')=2$ and (ii) $t_2\equiv 1 \pmod{4}, t'_2\equiv 3 \pmod{4}$. Then $\nu_2(s)-2\nu_2(t)=\nu_2(t)-2\nu_2(t')=-2$ and so $w_2^*(t)=-1$ and $w_2^*(t')=1$.
				\end{itemize}
				
			\item    If $\nu_2(s)\equiv 2 \pmod{4}, >2$, let $\ell=(\nu_2(s)-4)/2$. Choose $t, t'$ such that (i) $\nu_2(t)=\nu_2(t')=\ell$ and (ii) $t_2\equiv 1\pmod{4}, t'_2\equiv 3\pmod{4}$. Then $\nu_2(s)-2\nu_2(t)=\nu_2(s)-2\nu_2(t')=4$ and so $w_2^*(t)=-1$ and $w_2^*(t')=1$. \\
			
			\item    If $\nu_2(s)\equiv 3 \pmod{4}$, let $\ell=(\nu_2(s)-3)/2$. Choose $t, t'$ such that $\nu_2(t)=\nu_2(t')=\ell$. Then $\nu_2(s)-2\nu_2(t)=\nu_2(s)-2\nu_2(t')=3$ and
				\begin{itemize}
				\item    If $s_2\equiv 1\pmod{4}$, choose $t, t'$ such that (i) $\nu_2(t)=\nu_2(t')=\ell$ and (ii) $t_2\equiv 1 \pmod{8}, t'_2\equiv 3\pmod{8}$. Then $w_2^*(t)=-1$ and $w_2^*(t')=1$. 
				\item    If $s_2\equiv 3\pmod{4}$, choose $t, t'$ such that (i) $\nu_2(t)=\nu_2(t')=\ell$ and (ii) $t_2\equiv 1 \pmod{8}, t'_2\equiv 5\pmod{8}$. Then $w_2^*(t)=1$ and $w_2^*(t') =-1$. 
				\end{itemize}
			\end{enumerate}				
	\end{proof}

\subsection{Families $\mathcal{F}_s(t)$ with $t\in a\mathbb{Z}+b$} \label{sec: 3.2}
In this section we prove Theorem \ref{mainthm}, which we state again now. 

\begin{theorem}\label{thm: condition sab}(=Theorem \ref{mainthm})
Let $s, a, b$ be such that $s=-3r^2$ where $r\in \mathbb{Z}$. Then $\mathcal{F}_s(au+b)$ has constant root number if and only if
	\begin{enumerate}
	\item $\nu_p(b)<\nu_p(a)$ or $\nu_p(s)/2 \leq \nu_p(a) \leq \nu_p(b)$ for all $p\geq 5$ such that $p$ divides $s$, and
	\item $\nu_3(b)<\nu_3(a)$ or $(\nu_3(s)-3)/2 \leq \nu_3(a)\leq \nu_3(b)$, and
    \item the triple $(s,a,b)$ respects the 2-adic conditions of Table \ref{table1}. \end{enumerate}
\end{theorem}
		\begin{table}[H]
		\small
		\caption{}\label{table1}
		\centering
		\begin{tabular}{c c l c }
		\hline
		$\nu_2(a)-\nu_2(b)$	&$\nu_2(s)\pmod 4$&	$\nu_2(s)-2\nu_2(a)$ &$b_2\pmod 4$\\
		\hline
		$>2$		&		&								\\
		2				&0	&$<-4$					&		\\
						&		&$2 \pmod 4, >-4$	&		\\
						&		&$0\pmod 4, >-4$	&1		\\
						&2	&$<-4$					&		\\
						&		&$-2$						&3	\\
						&		&$0\pmod 4, >-4$	&		\\
						&		&$2$						&		\\
						&		&$2\pmod 4, >2$		&1		\\
		$\leq 1$	&0	&$\leq -6$				&		\\
						&2	&$\leq -6$				&		\\
						&		&$-4$						&		\\
						\hline
		\end{tabular}
		\end{table}

\begin{proof}
First note that the assumption of $s=-3r^2$ implies that $\nu_p(s)\equiv 0 \pmod 2$ for $p\geq 5$, $\nu_3(s)\equiv 1 \pmod 2$, and $\nu_2(s)\equiv 0 \pmod 2$. 
By Theorem \ref{thm: indep}, we have $W(\mathcal{F}_s(au+b))=-\prod_{p<\infty}{w_p^*(au+b)}$ where $w_p^*(au+b)$ is as defined in \eqref{eq: w_p^*}. If $p\nmid s$, then $\nu_p(s)=0$ and so $2\nu_p(t)-\nu_p(s)\geq 0$. By Proposition \ref{prop: w_p^*(t)} Table \ref{table: w_p^*(t)}, if $2\nu_p(t)-\nu_p(s)> 0$, then $w_p^*(t)=1$ and if $2\nu_p(t)-\nu_p(s)=0$, then
    \begin{align*}
        w_p^*(t)=\begin{cases}
        (\frac{-3}{p}),&\nu_p(t^2-s)\equiv 2, 4 \pmod 6\\
        1,& \text{otherwise}.
        \end{cases}
    \end{align*}
In the case that $\nu_p(t^2-s)=0$ we have $w_p^*(t)=1$. Otherwise, $0\equiv t^2-s \equiv t^2+3r^2\pmod p$ and hence $w_p^*(t)=(\frac{-3}{p})=1$. Thus we have shown that
    \begin{align}
        W(\mathcal{F}_s(au+b))=-\prod_{p\mid s}{w_p^*(au+b)}
    \end{align}
Recall that these functions are independent from one another, which means that the root number $W\left(\mathcal{F}_s(au+b)\right)$ is constant over $\Z$ if and only if $w_p^*(au+b)$ is constant for each $p$. We split the proof into Lemmas \ref{lem: p5}, \ref{lem: p=3} and \ref{lem: p=2} which study $w_p^*(au+b)$ respectively when $p\geq5$, $p=3$ and $p=2$.
\end{proof}

\begin{rem}
	In order to find the value of $w_p^*(t)$ using the Propositions \ref{prop: w_p^*(t)}, \ref{prop: w_3^*(t)} and \ref{prop: w_2^*(t)}, we need to take into account some $p$-adic information such as $\nu_p(s)$, $2\nu_p(t)-\nu_p(s)$, $\nu_p(t^2-s)$, $t_p$, and others. It will be more convenient to order the cases with respect to whether $\nu_p(a)$ or $\nu_p(b)$ is higher, since if we
		let $m=\min(\nu_p(a),\nu_p(b))$, then
	\begin{align*}
	t=au+b&=p^{m}(p^{\nu_p(a)-m}a_pu+p^{\nu_p(b)-m}b_p)\\
	&=\begin{cases}
	p^{\nu_p(b)}(p^{\nu_p(a)-\nu_p(b)}a_pu+b_p), \nu_p(b)\leq \nu_p(a)\\
	p^{\nu_p(a)}(a_pu+p^{\nu_p(b)-\nu_p(a)} b_p), \nu_p(a)\leq \nu_p(b)
	\end{cases}
	\end{align*}
	and so we obtain $\nu_p(t)$ as follows:
	\begin{align*}
	\nu_p(t)=\begin{cases}
	\nu_p(b), &\nu_p(b)<\nu_p(a)\\
	\nu_p(a)+\nu_p(a_pu+p^{\nu_p(b)-\nu_p(a)} b_p), &\nu_p(a)\leq \nu_p(b).
	\end{cases}
	\end{align*}
\end{rem}

\subsubsection{Proof of Theorem \ref{thm: condition sab} for $p\geq 5$}

\begin{lemma}\label{lem: p5}
Let $s,a,b$ be such that $s=-3r^2$ where $r\in\Z$, and let $p\geq 5$ be a prime number such that $p\mid r$. Then the function $w_p^*(au+b)$ is constant as $u$ varies through $\Z$ if and only if one of the following holds:
\begin{enumerate}
    \item $\nu_p(b)<\nu_p(a)$
    \item $\nu_p(s)/2\leq\nu_p(a)\leq\nu_p(b)$.
\end{enumerate}
\end{lemma}
\begin{proof}
We refer to Proposition \ref{prop: w_p^*(t)} Table \ref{table: w_p^*(t)} for most of the proof. Note that we have $\nu_p(s)\equiv 0 \pmod 2, >0$ and if $p\mid t^2-s$, then $(\frac{-3}{p})=1$.
	\begin{enumerate}[wide=0pt,listparindent=1.5em]
	\item 
	If $\nu_p(b)<\nu_p(a)$, then $$t=au+b= p^{\nu_p(b)}(p^{\nu_p(a)-\nu_p(b)}a_pu+b_p)$$
	where $p\nmid p^{\nu_p(a)-\nu_p(b)}a_pu+b_p$, so $$\nu_p(t)=\nu_p(b) \text{ and } t_p\equiv b_p \pmod{p}.$$
	 If $\nu_p(s)\equiv 0 \pmod{4}$, then for case $2\nu_p(t)=\nu_p(s)$, we have $p\mid s$ and $p \mid t$ and in particular $p\mid t^2-s$. Then $(\frac{-3}{p})=1$ and so $w_p^*(t)=1$. If $\nu_p(s)\equiv 2 \pmod{4}$, then for case $2\nu_p(t)=\nu_p(s)$, again we have $p\mid t^2-s$, so $(\frac{3}{p})=(\frac{-3}{p})(\frac{-1}{p})=(\frac{-1}{p})$, and hence $w_p^*(t)=(\frac{-1}{p})$. So this part of Table \ref{table: w_p^*(t)} simplifies to
	    \begin{align*}
	        w_p^*(t)=\begin{cases}
	           \begin{cases}
	           -(\frac{3b_p}{p}), &\nu_p(b)\equiv 0 \pmod 2\\
	           (\frac{-1}{p}), &\nu_p(b)\equiv 1 \pmod 2
	           \end{cases}, &2\nu_p(b)-\nu_p(s)<0    \\       \begin{cases}
	           1, &\nu_p(s)\equiv 0 \pmod 4\\
	           (\frac{-1}{p}), &\nu_p(s)\equiv 2 \pmod 4
	           \end{cases}, &2\nu_p(b)-\nu_p(s)\geq0       
	        \end{cases}
	    \end{align*}
%
%
where we see that $w_p^*(t)$ does not depend on $t$ and hence is constant.\\
		
	  \item 
	  If $\nu_p(a)\leq \nu_p(b)$ and $\nu_p(s)/2 \leq \nu_p(a)$, then firstly $\nu_p(a)\leq \nu_p(b)$ means 					$$t=au+b= p^{\nu_p(a)}(a_pu+p^{\nu_p(b)-\nu_p(a)}b_p)$$
		and so $\nu_p(t)\geq \nu_p(a)$, since $p$ could divide $a_pu+p^{\nu_p(b)-\nu_p(a)}b_p$ depending on whether $p\mid u$. So we have $\nu_p(s)/2\leq \nu_p(a)\leq \nu_p(t)$, and hence $2\nu_p(b)-\nu_p(s)=2\nu_p(t)-\nu_p(s)\geq 0$. 
		The rest of the argument follows analogously to the one above, where now Table \ref{table: w_p^*(t)} simplifies to 
		    \begin{align*}
		        w_p^*(t) = \begin{cases}
		        1, &\nu_p(s)\equiv 0 \pmod 4\\
		        (\frac{-1}{p}), &\nu_p(s)\equiv 2 \pmod 4.
		        \end{cases}
		    \end{align*}
		Again we see that $w_p^*(t)$ does not depend on $t$ and hence is constant.\\
			
\item If $\nu_p(a)\leq \nu_p(b)$ and $\nu_p(a)<\nu_p(s)/2$, then $2\nu_p(a)<\nu_p(s)$ and again $\nu_p(a)\leq \nu_p(b)$ implies $2\nu_p(a)\leq 2\nu_p(t)$. We find $t, t'$ with different local root numbers at $p$. 		
\begin{itemize}
 			\item    
 If $\nu_p(s)\equiv 0 \pmod 4$, first choose $t$ such that $\nu_p(t)>\nu_p(s)/2$. Then $2\nu_p(t)-\nu_p(s)>0$, and so $w_p^*(t)=1$. Now we choose $t'$ as follows:
			\begin{align*}
			    t'=
			    \begin{cases}
			    \begin{cases}
			    p^{\nu_p(a)}\alpha, &p\equiv \pm 1 \pmod {12} \\
			    p^{\nu_p(a)}\beta, &p\equiv \pm 5 \pmod {12} \\
			    \end{cases}, &\nu_p(a)\equiv 0 \pmod 2\\
			    \begin{cases}
			   	    p^{\nu_p(a)-1}\alpha, &p\equiv \pm 1 \pmod {12}\\
			        p^{\nu_p(a)-1}\beta, &p\equiv \pm 5 \pmod {12}\\
			    \end{cases}, &\nu_p(a)\equiv 1 \pmod 2\\
			    \end{cases}
			\end{align*}
			where $\alpha,\beta\in\Z$ are both not divisible by $p$ and such that $\alpha$ is a quadratic residue mod $p$ while $\beta$ is a quadratic nonresidue mod $p$. We claim that these choices of $t'$ give $w_p^*(t)=-1$.
			If $\nu_p(a)\equiv 0 \pmod2$, then choosing $\nu_p(t')=\nu_p(a)$ means $2\nu_p(t')-\nu_p(s)<0$ and $\nu_p(t')\equiv 0 \pmod 2$ which imply $w_p^*(t')=-(\frac{3t'_p}{p})$. If $\nu_p(a)\equiv 1 \pmod2$, then choosing $\nu_p(t')=\nu_p(a)-1$ again gives $w_p^*(t')=-(\frac{3t'_p}{p})$. In both cases we have $w_p^*(t')=-(\frac{3t'_p}{p})$, and we want this to be $-1$. Note that 
			$$-\left(\frac{3t'_p}{p}\right)=-\left(\frac{3}{p}\right)\left(\frac{t'_p}{p}\right)=\begin{cases}
			-\left(\frac{t'_p}{p}\right), \ &p\equiv \pm 1 \pmod{12}\\
			\left(\frac{t'_p}{p}\right), \ &p \equiv \pm 5 \pmod{12}
			\end{cases}$$
			so choose $t'_p$ to be a quadratic residue or quadratic nonresidue mod $p$ according to $p$ mod $12$.

			\item If $\nu_p(s)\equiv 2 \pmod 4$, the process is quite similar. First choose $t$ such that $\nu_p(t)>\nu_p(s)/2$, so $w_p^*(t)=(\frac{-1}{p})$. Now choose $t'$ as follows:
            \begin{align*}
			    t'=
			    \begin{cases}
			    \begin{cases}
			    p^{\nu_p(a)}\alpha, &\text{if }p\equiv 1 \pmod 3 \\
			    p^{\nu_p(a)}\beta, &\text{if }p\equiv 2 \pmod 3 \\
			    \end{cases}, &\text{and }\nu_p(a)\equiv 0 \pmod 2\\
			    \begin{cases}
			   	  p^{\nu_p(a)-1}\alpha, &\text{if }p\equiv 1 \pmod 3\\
			      p^{\nu_p(a)-1}\beta, &\text{if }p\equiv 2 \pmod 3
			    \end{cases}, &\text{and }\nu_p(a)\equiv 1 \pmod 2
			    \end{cases}
			\end{align*}
			where again $\alpha,\beta\in\Z$ are both not divisible by $p$ and such that $\alpha$ is a quadratic residue mod $p$ while $\beta$ is a quadratic nonresidue mod $p$.
			We claim that $w_p^*(t')=-(\frac{-1}{p})$. If $\nu_p(a)\equiv 0 \pmod2$, then choosing $\nu_p(t')=\nu_p(a)$ means $2\nu_p(t')-\nu_p(s)<0$ and $\nu_p(t')\equiv 0 \pmod 2$ which imply $w_p^*(t')=-(\frac{3t'_p}{p})$. If $\nu_p(a)\equiv 1 \pmod2$, then choosing $\nu_p(t')=v_p(a)-1$ again gives $w_p^*(t')=-(\frac{3t'_p}{p})$. Note that
			$-(\frac{3t'_p}{p})=-(\frac{-1}{p})(\frac{-3}{p})(\frac{t'_p}{p})$	
			and we want $(\frac{-3}{p})(\frac{t'_p}{p})=1$. So choose $t'_p$ to be a quadratic residue or quadratic nonresidue mod $p$ according to $p$ mod $3$. \qedhere 
					\end{itemize}
		\end{enumerate}
		\end{proof}
		
\subsubsection{Proof of Theorem for $p=3$}\label{sec:mainthm_pf_p=3}

\begin{lemma}\label{lem: p=3}
Let $s,a,b\in\Z$ be such that $s=-3r^2$ where $r\in\Z$. Then the function $w_3^*(au+b)$ is constant as $u$ varies through $\Z$ if and only if one of the following holds:
\begin{enumerate}
    \item $\nu_3(b)<\nu_3(a)$ or
    \item $(\nu_3(s)-3)/2\leq\nu_3(a)\leq\nu_3(b)$.
\end{enumerate}
\end{lemma}
\begin{proof}
We have $\nu_3(s)\equiv 1 \pmod 2$, so we refer to Proposition \ref{prop: w_3^*(t)} Table \ref{table: w_3^*(t),v_3(s)odd}.
		\begin{enumerate}[wide=0pt]
		\item If $\nu_3(b)<\nu_3(a)$, then as above,
	$$t=au+b= 3^{\nu_3(b)}(3^{\nu_3(a)-\nu_3(b)}a_3u+b_3)$$
	where $3\nmid 3^{\nu_3(a)-\nu_3(b)}a_3u+b_3$, so 
	$\nu_3(t)=\nu_3(b) \text{ and } t_3\equiv b_3 \pmod{3}$.
Replacing $\nu_3(t)$ by $\nu_3(b)$ and $t_3$ by $b_3$ in Table \ref{table: w_3^*(t),v_3(s)odd}, we get
		\begin{table}[H]
		\centering
		\caption{} \label{table:p=3_rank}
		\begin{tabular}{l l r} 
		\hline
		$\nu_3(s)$	&$\nu_3(s)-2\nu_3(b)$	&$w_3^*(t)$\\
		\hline
		$1\pmod 4$	&$-1$						&$-(\frac{s_3}{3})$\\
			&$1\pmod{4}, >1$		&$-1$\\
			&$3\pmod{4}, >3$	&$-(\frac{b_3}{3})$\\
			&otherwise  &1\\
		$3\pmod 4$	&$1$							&$(\frac{s_3}{3})$\\
			&$1\pmod{4}, >1$		&$-(\frac{b_3}{3})$\\
			&otherwise				&$-1$\\
			\hline
		\end{tabular}
		\end{table}
		Then $w_3^*(t)$ is independent of $t$ and hence is constant given $s, a, b$.\\
	 
		\item If $\nu_3(a)\leq \nu_3(b)$ and $(\nu_3(s)-3)/2 \leq \nu_3(a)$, then
		$$t=au+b= 3^{\nu_3(a)}(a_3u+3^{\nu_3(b)-\nu_3(a)}b_3)$$
		and so $\nu_3(t)\geq \nu_3(a)$, since 3 could divide $a_3u+3^{\nu_3(b)-\nu_3(a)}b_3$ depending on whether $3\mid u$.
		Then $(\nu_3(s)-3)/2 \leq \nu_3(a) \leq \nu_3(t)$ and so $\nu_3(s)-2\nu_3(t)\leq 3$.
		By Table \ref{table: w_3^*(t),v_3(s)odd}, we see that $w_3^*(t)$ is independent of $t$ for both $\nu_3(s)\equiv 1 \pmod 4$ and $\nu_3(s)\equiv 3 \pmod 4$.		\\
		\item If $\nu_3(a)\leq \nu_3(b)$ and $\nu_3(a)<(\nu_3(s)-3)/2$, we find $t, t'$ with different root numbers at 3. The first assumption implies $\nu_3(t)\geq \nu_3(a)$ and so $2\nu_3(a)+3 \leq 2\nu_3(t)+3$. The second assumption is $2\nu_3(a)+3<\nu_3(s)$.

	\begin{itemize}
		\item If $\nu_3(s)\equiv 1 \pmod 4$, first choose $t$ such that $\nu_3(t)>(\nu_3(s)-3)/2$. Then $\nu_3(s)-2\nu_3(t)<3$ and so $w_3^*(t)=1$. Note that if $\nu_3(s)-2\nu_3(t)=-1$, then $w_3^*(t)=-(\frac{s_3}{3})=1$, since $s_3=-r_3^2\equiv -1\pmod{3}$. Now choose $t'$ such that
		    \begin{align*}
		        t'=3^{\nu_3(a)}\alpha
		    \end{align*}
		for some $\alpha\equiv 1 \pmod 3$. In other words
		$\nu_3(t')=\nu_3(a)$ and $t'_3\equiv 1 \pmod{3}$. Then $2\nu_3(t')+3=2\nu_3(a)+3< \nu_3(s)$, so $\nu_3(s)-2\nu_3(t')> 3$. If $\nu_3(s)-2\nu_3(t')\equiv 1 \pmod{4}$, then $w_3^*(t)=-1$, and if $\nu_3(s)-2\nu_3(t')\equiv 3 \pmod 4$, then $w_3^*(t) = -(\frac{t'_3}{3})=-1$.
			
		\item If $\nu_3(s)\equiv 3 \pmod 4$, again choose $t$ such that $\nu_3(t)>(\nu_3(s)-3)/2$. Then $\nu_3(s)-2\nu_3(t)<3$ and so $w_3^*(t)=-1$. Now choose $t'$ as follows:
		    \begin{align*}
		        t'=\begin{cases}
		        3^{\nu_3(a)}\beta, &\nu_3(s)-2\nu_3(a)\equiv 1 \pmod 4\\
		        3^{\nu_3(a)+1}\beta, &\nu_3(s)-2\nu_3(a)\equiv 3 \pmod 4
		        \end{cases}
		    \end{align*}
		    for some $\beta\equiv 2 \pmod 3$. 
		    We refer to Table \ref{table: w_3^*(t),v_3(s)odd} to deduce that $w_p^*(t')=1$. \qedhere
        \end{itemize}	
		\end{enumerate}
\end{proof}		
\subsubsection{Proof of Theorem \ref{thm: condition sab} for $p=2$}

\begin{lemma}\label{lem: p=2}
Let $s,a,b$ be such that $s=-3r^2$ where $r\in\Z$. Then the function $w_p^*(au+b)$ is constant as $u$ varies through $\Z$ if and only if the triple $(s,a,b)$ respects the $2$-adic data of Table \ref{table1}.
\end{lemma}
\begin{proof}
For $p=2$, we have $\nu_2(s)\equiv 0 \pmod 2$, so we refer to Proposition \ref{prop: w_2^*(t)} Tables \ref{table: w_2^*(t),v_2(s)0mod4}, \ref{table: w_2(t) for 2v_2(t)=v(s),v_2(t)even}, \ref{table: w_2^*(t),v_2(s)2mod4} and \ref{table: w_2(t) for 2v_2(t)=v(s),v_2(t)odd} for most of the proof.
		\begin{enumerate}[wide=0pt]
		\item If $\nu_2(b)+2<\nu_2(a)$, then $$t=au+b= 2^{\nu_2(b)}(2^{\nu_2(a)-\nu_2(b)}a_2u+b_2)$$
	where $2\nmid 2^{\nu_2(a)-\nu_2(b)}a_2u+b_2$ and so $\nu_2(t)=\nu_2(b) \text{ and } t_2\equiv b_2\pmod{8}$.
	For the case $2\nu_2(b)\neq \nu_2(s)$, refer to Proposition \ref{prop: w_2^*(t)} Tables \ref{table: w_2^*(t),v_2(s)0mod4} and \ref{table: w_2^*(t),v_2(s)2mod4} for $\nu_2(s)\equiv 0, 2 \pmod{4}$ and replace $\nu_2(t)$ by $\nu_2(b)$ and $t_2$ by $b_2$ to get the roots numbers $w_2^*(t)$ which do not depend on $t$. For the case $2\nu_2(b)=\nu_2(s)$, we have
	\begin{align*}
	t^2-s
	&=2^{2\nu_2(t)}(t_2^2-s_2)=2^{2\nu_2(t)}(2^{2(\nu_2(a)-\nu_2(b))}a_2^2u^2+2^{\nu_2(a)-\nu_2(b)+1}a_2b_2u+b_2^2-s_2).
	\end{align*}
	Note that $2(\nu_2(a)-\nu_2(b))\geq 6$ and $\nu_2(a)-\nu_2(b)+1\geq 4$. Further we have $b_2^2\equiv 1 \pmod 8$ and $s_2 = -3r_2^2 \equiv 5 \pmod 8$ and so $b_2^2-s_2\equiv 4 \pmod 8$ which means $\nu_2(b_2^2-s_2)=2$.  
	Then
	\begin{align}\label{eqn:p=2_case_1}
	t^2-s&=2^{2\nu_2(t)+2}(2^{2(\nu_2(a)-\nu_2(b))-2}a_2^2u^2+2^{\nu_2(a)-\nu_2(b)-1}a_2b_2u+(b_2^2-s_2)_2)
	\end{align}
	so that $\nu_2(t^2-s)-2\nu_2(t)=2$. If $\nu_2(t)$ even, then by Table \ref{table: w_2(t) for 2v_2(t)=v(s),v_2(t)even} we have
	$$w_2^*(t)=\begin{cases}
	1, \ &t_2(t^2-s)_2\equiv 3\pmod{4}\\
	-1, \ &\text{otherwise}
	\end{cases}$$
	and if $\nu_2(t)$ odd, then by Table \ref{table: w_2(t) for 2v_2(t)=v(s),v_2(t)odd} we have
	$$w_2^*(t)=\begin{cases}
	1,  \ &t_2\equiv 1\pmod{4}\\
	& t_2\equiv 3\pmod{4}\text{ and }(t^2-s)_2\equiv3\pmod{8}\\
	-1, \ &\text{otherwise.}
	\end{cases}$$
	In the above expression \eqref{eqn:p=2_case_1}, note that $2(\nu_2(a)-\nu_2(b)-2)\geq 2$ and $\nu_2(a)-\nu_2(b)-1\geq 2$, so $(t^2-s)_2\equiv (b_2^2-s_2)_2 \pmod 4$. 	These formulas thus do not depend on $t$. This argument works for $\nu_2(b)+2=\nu_2(a)$ as well.\\

	\item If $\nu_2(b)+2=\nu_2(a)$, then $$t=au+b= 2^{\nu_2(b)}(2^2a_2u+b_2)$$
	and so $\nu_2(t)=\nu_2(b) \text{ and } t_2\equiv b_2\pmod{4}$. If $2\nu_2(b)= \nu_2(s)$, then we use the the same argument as above. If $2\nu_2(b)\neq \nu_2(s)$, then we consider the following.
			\begin{itemize}
			\item If $\nu_2(s)\equiv 0 \pmod 4$ we refer to Table \ref{table: w_2^*(t),v_2(s)0mod4}. If $\nu_2(s)-2\nu_2(b)<0$ or $\nu_2(s)-2\nu_2(b)\equiv 2 \pmod4$, then $w_2^*(t)$ is constant by replacing $\nu_2(t)$ by $\nu_2(b)$ and $t_2$ by $b_2$. Now suppose $\nu_2(s)-2\nu_2(b)\equiv 0 \pmod 4, >0$. If $b_2\equiv 1 \pmod 4$, then we have $t_2\equiv 1$ or $5 \pmod{8}$ and so 
			    \begin{align}\label{eqn:p=2_case_2}
			        w_2^*(t)=\begin{cases}
			            1,& \nu_2(s)-2\nu_2(b)=4\\
			            -1,& \nu_2(s)-2\nu_2(b)\equiv 0 \pmod 4, >4.
			        \end{cases}
			    \end{align}    
			 If $b_2\equiv 3 \pmod 4$, then choose $t$ and $t'$ such that $t_2\equiv 3 \pmod 8$ and $t'_2\equiv 7 \pmod 8$. Then $w_2^*(t)$ is as \eqref{eqn:p=2_case_2} while $w_2^*(t')=-w_2^*(t)$.

			\item If $\nu_2(s)\equiv 2 \pmod 4$ we refer to Table \ref{table: w_2^*(t),v_2(s)2mod4}. 
				\begin{itemize}

				\item    If $\nu_2(s)-2\nu_2(b)=2$, then if $b_2\equiv 3 \pmod 4$, we have $t_2\equiv 3, 7 \pmod{8}$ and so $w_2^*(t)=1$. But if $b_2\equiv 1 \pmod{4}$, we can choose $t$ and $t'$ such that $t_2\equiv 5 \pmod 8$ and $t'_2\equiv 1 \pmod 8$. Then for $s_2\equiv 1 \pmod 8$, $w_2^*(t)=1$ and $w_2^*(t')=-1$. And for $s_2\equiv 5 \pmod 8$, $w_2^*(t)=-1$ and $w_2^*(t')=1$.
				
				\item    If $\nu_2(s)-2\nu_2(b)\equiv 2 \pmod 4, >6$, then if $b_2\equiv 1 \pmod 4$, we get $t_2\equiv 1, 5 \pmod 8$ and so $w_2^*(t)=-1$. If $b_2\equiv 3 \pmod 4$, then choose $t$ and $t'$ such that $t_2\equiv 7 \pmod 8$ and $t'_2\equiv 3 \pmod 8$. Then $w_2^*(t)=1$ and $w_2^*(t')=-1$.\\
				\end{itemize}							
	\end{itemize}
	
	\item If $\nu_2(b)+1=\nu_2(a)$, then 
		\begin{align*}
	        t=au+b=2^{\nu_2(b)}(2a_2u+b_2)
        \end{align*}
	so $\nu_2(t)=\nu_2(b)$ and $t_2\equiv b_2\pmod 2$.
		For $\nu_2(s)-2\nu_2(b)\leq -4$, $w_2^*(t)$ is constant, again by replacing $\nu_2(t)$ by $\nu_2(b)$ in Tables \ref{table: w_2^*(t),v_2(s)0mod4} and \ref{table: w_2^*(t),v_2(s)2mod4}. Now suppose $\nu_2(s)-2\nu_2(b)>-4$.
			\begin{itemize}
			\item    If $\nu_2(s)\equiv 0 \pmod 4$, then choose the following $t$ and $t'$ to get different local root numbers:
				\begin{center}
				\begin{tabular}{l l l l r r}
				\hline
				$\nu_2(s)-2\nu_2(b)$	&$s_2$&$t_2$	&$t'_2$	&$w_2^*(t)$	&$w_2^*(t')$\\
				\hline
				$-2$						&$1, 13 \pmod{16}$	&$3\pmod 4$		&$1\pmod 4$			&$-1$	&$1$\\
											&$5, 9 \pmod{16}$	& $3\pmod 4$		&$1\pmod 4$			&$1$		&$-1$\\
				2							&								&$1\pmod{4}$	&$3\pmod{4}$	&$1$		&$-1$\\
				$2\pmod{4}, >2$	&								&$3\pmod{4}$		&$1\pmod{4}$			&$1$		&$-1$\\
				4							& 								&$3\pmod{8}$		&$7\pmod{8}$		&$1$		&$-1$\\
				$0\pmod{4}, >4$	&								&$3\pmod{8}$		&$7\pmod{8}$		&$-1$	&$1$\\
				\hline
				\end{tabular}
				\end{center}
				
			\item    If $\nu_2(s)\equiv 2 \pmod 4$, then for $\nu_2(s)-2\nu_2(b)=-2$, we have $w_2^*(t)=-1$ since $s_2\equiv 1 \pmod 4$. For the remaining except $\nu_2(s)-2\nu_2(b)=0$, choose the following $t$ and $t'$ to get different local root numbers:
				\begin{center}
				\begin{tabular}{l l l r r}
				\hline
				$\nu_2(s)-2\nu_2(b)$	&$t_2$	&$t'_2$	&$w_2^*(t)$	&$w_2^*(t')$\\
				\hline
				$0\pmod{4}, >0$						&$3\pmod{4}$	&$1\pmod{4}$	&$1$		&$-1$\\
				2						
					&$1\pmod{4}$	&$3\pmod{4}$	&$1$		&$-1$\\
				6 						&$3\pmod{4}$	&$1\pmod{4}$	&$1$		&$-1$\\
				$2\pmod 4, >6$			&$7\pmod 8$	&$1\pmod 8$				&1			&$-1$\\
				\hline
				\end{tabular}
				\end{center}

	 Finally if $2\nu_2(b)=\nu_2(s)$, then $2\nu_2(t)=\nu_2(s)$ and
	\begin{align*}
	t^2-s&=2^{2\nu_2(t)}(t_2^2-s_2)\\
	&=2^{2\nu_2(t)}(2^{2}a_2^2u^2+2^{2}a_2b_2u+b_2^2-s_2)
	\end{align*}
	Again we have $b_2^2\equiv 1 \pmod 8$ and $s_2=-3r_2^2\equiv 5 \pmod 8$, so $b_2^2-s_2\equiv 4 \pmod 8$ and hence $\nu_2(b_2^2-s_2)=2$. Then
	\begin{align*}
	t^2-s&=2^{2\nu_2(t)+2}(a_2^2u^2+a_2b_2u+(b_2^2-s_2)_2).
	\end{align*}
	
	By choosing $u$ such that $\nu_2(u)=1$, we have 
$$\nu_2(t^2-s)-2\nu_2(t)=2$$ and 
$$(t^2-s)_2=a_2^2u^2+a_2b_2u+(b_2^2-s_2)_2.$$
Then for the case $\nu_2(b)$ even, using Table \ref{table: w_2(t) for 2v_2(t)=v(s),v_2(t)even} and $t_2=2a_2u+b_2$, we compute $$t_2(t^2-s)_2=(2a_2u+b_2)(a_2^2u^2+a_2b_2u+(b_2^2-s_2)_2).$$ If $u$ is even, we have 
	\begin{align*}
	 t_2(t^2-s)_2&\equiv b_2(a_2b_2u+(b_2^2-s_2)_2) \pmod 4\\
	 &\equiv a_2u+b_2(b_2^2-s_2)_2 \pmod 4
	 \end{align*}	
	 where $b_2(b_2^2-s_2)_2\equiv \pm 1 \pmod 4$ and $a_2u\equiv 0$ or $2 \pmod 4$.
	 So we can choose $u\equiv 0 \pmod 4$ and $u'\equiv 2 \pmod 4$ so that $t_2(t^2-s)_2$ and $t'_2(t'^2-s)_2$ have different signs mod 4 and hence different root numbers according to Table \ref{table: w_2(t) for 2v_2(t)=v(s),v_2(t)even}.\\
	 
	 On the other hand, if $\nu_2(b)$ is odd, then we refer to Table \ref{table: w_2(t) for 2v_2(t)=v(s),v_2(t)odd}. Choose $u, u'$ such that
	\begin{align*}
	 u\equiv \frac{(1-b_2)}{2a_2}\pmod 4 \ \text{and } u'\equiv \frac{(-1-b_2)}{2a_2}\pmod 8
\end{align*}

Then $2u\equiv a_2^{-1}(1-b_2)\pmod 4$ and so $t_2=2a_2u+b_2\equiv 1 \pmod 4$, and so $w_2^*(t)=1$. Similarly, $2u'\equiv a_2^{-1}(-1-b_2)\pmod 8$, and so $t_2'=2a_2u'+b_2\equiv -1 \equiv 7 \pmod 8$, and so $w_2^*(t')=-1$.\\
\end{itemize}
		\item If $\nu_2(a)\leq \nu_2(b)$ and $(\nu_2(s)+6)/2 \leq \nu_2(a)$, then 
				$$t=au+b= 2^{\nu_2(a)}(a_2u+2^{\nu_2(b)-\nu_2(a)}b_2)$$
		so that $\nu_2(t)\geq \nu_2(a)$. Then $\nu_2(s)+6\leq 2\nu_2(a)\leq 2\nu_2(t)$ and $\nu_2(s)-2\nu_2(t)\leq -6$. Further, note that $s_2=-3r_2^2\equiv 1 \pmod4$.
		Thus for $\nu_2(s)\equiv 0 \pmod 4$, Table \ref{table: w_2^*(t),v_2(s)0mod4} gives
		    \begin{align*}
		        w_2^*(t)=\begin{cases} 1, \ &s_2\equiv 5, 9 \pmod{16}\\
			-1, \ &s_2\equiv 1, 13 \pmod{16}
			\end{cases}
		    \end{align*}
		which is independent of $t$. For $\nu_2(s)\equiv 2 \pmod 4$ Table \ref{table: w_2^*(t),v_2(s)2mod4} gives $w_2^*(t)=s_2\pmod 4=1$.\\
	%
		\item If $\nu_2(a)\leq \nu_2(b)$ and $(\nu_2(s)+4)/2= \nu_2(a)$, then again $\nu_2(a)\leq \nu_2(t)$, so $\nu_2(s)-2\nu_2(t)\leq -4$. For $\nu_2(s)\equiv 0\pmod 4$, we find $t, t'$ with different local root numbers at $2$. First choose $t$ such that $\nu_2(s)-2\nu_2(t)=-4$. Then by Table \ref{table: w_2^*(t),v_2(s)0mod4},
		    \begin{align*}
		        w_2^*(t)=\begin{cases} -1, \ &s_2\equiv 5, 9 \pmod{16}\\
			1, \ &s_2\equiv 1, 13 \pmod{16}.
			\end{cases}
		    \end{align*}
		Now choose $t'$ such that $\nu_2(s)-2\nu_2(t')<-4$ so $w_2^*(t') = -w_2^*(t)$.
		For $\nu_2(s)\equiv 2\pmod4$, $w_2^*(t)\equiv s_2\pmod 4=1$ is constant.\\
			
		\item If $\nu_2(a)\leq \nu_2(b)$ and $\nu_2(a)< (\nu_2(s)+4)/2$, then $2\nu_2(a)\leq 2\nu_2(t)$ and $2\nu_2(a)<\nu_2(s)+4$. 
			\begin{itemize}
			\item For $\nu_2(s)\equiv 0 \pmod 4$, choose $t, t'$ such that $\nu_2(t)=\nu_2(t')=(\nu_2(s)+2)/2 \geq \nu_2(a)$. Then $\nu_2(s)-2\nu_2(t)=-2$. Further, let $t_2\equiv 3 \pmod 4$ and $t'_2\equiv 1 \pmod4$. Then for $s_2\equiv 1, 13 \pmod{16}$, $w_2^*(t)=-1$ and $w_2^*(t')=1$. And for $s_2\equiv 5, 9 \pmod{16}$, $w_2^*(t)=1$ and $w_2^*(t')=-1$. 
			\item    For $\nu_2(s)\equiv 2 \pmod 4$, choose $t$ such that $\nu_2(t)=(\nu_2(s)+2)/2\geq \nu_2(a)$. Then $\nu_2(s)-2\nu_2(t)=-2$, so $w_2^*(t)=-1$. Next choose $t'$ such that $\nu_2(t')>(\nu_2(s)+4)/2$. Then $\nu_2(s)-2\nu_2(t')<-4$, so $w_2^*(t')\equiv s_2\pmod 4 =1$. \qedhere
			\end{itemize}	
			\end{enumerate}

\end{proof}
\subsection{Families $\mathcal{L}_{v,s,w}(t)$, $t\in a\Z+b$}\label{sec: 3.3}

We prove in this section Corollary \ref{cor: main_to_L} and justify that the family given in Example \ref{ex: notcomingfromconstant} has indeed constant root number over $\Z$ (it is covered by the more general situation of Lemma \ref{mainthmL}).

\begin{theorem}\label{thm: basechange}
Let $r,v,w\in\Q$ non-zero, such that $s=-3r^2w^2$, $a=w$, $b=wv$ respect the conditions in Theorem \ref{thm: condition sab}. Then $\mathcal{F}_{s}(wt+wv)$ has constant root number over $\Z$. 
In particular any linear subfamily of $\mathcal{L}_{w,-3r^2,v}(t)$ (equivalent to $\mathcal{F}_{-3r^2w^2}(wt^2+wv)$) has constant root number over $\Z$.
\end{theorem}

\begin{proof}As observed in Section \ref{section:familyL}, a family $\mathcal{L}_{w,s,v}(t)$ can be obtained by quadratic base change from a family $\mathcal{F}_{sw^2}(t)$ via $t:=wt^2+vw$. It follows that if the root number is constant on the integer fibres of $\mathcal{F}_{sw^2}(w(au+b)^2+wv)$, then it is also constant on the integer fibres of $\mathcal{L}_{w,s,v}(au+b)$ for any $a,b,v\in\Z$.

We study $T(t)=wt^2+wv$. By setting $A=w$, $B=wv$, $U(u)=(au+b)^2$ and $S=sw^2$ and observing that $T(u)=AU(u)+B$, we are able to use Theorem \ref{thm: condition sab} to evaluate the root number of $\mathcal{F}_{sw^2}(wt+wv)$ and this is equivalent to $\mathcal{L}_{w,s,v}(t)$ (as is stated in Lemma \ref{lem: familyL}).

The only thing left in order to prove Corollary \ref{cor: main_to_L} is to verify that it corresponds to the conditions listed in the introduction. Suppose that $p\geq5$. We need to have $p\mid rw$, else $p\nmid S$. Moreover, if $S,A,B$ respect $\nu_p(B)<\nu_p(A)$, then we have $\nu_p(w)+\nu_p(v)<\nu_p(w)$, thus $\nu_p(v)<0$. If $S,A,B$ respect $\nu_p(S)/2\leq\nu_p(A)\leq\nu_p(B)$, then $\nu_p(r)+\nu_p(w)\leq\nu_p(w)\leq\nu_p(w)+\nu_p(v)$ and thus $\nu_p(r)\leq0\leq\nu_p(v)$. The conditions for $p=2,3$ are found similarly.
\end{proof}

However, the other direction is not true i.e. an $\mathcal{L}$-family  may have constant root number, even if it comes from a $\mathcal{F}$-family with non-constant root number, as shown in the following example.

\begin{exam}
Consider the family $\mathcal{L}_{7,-3\cdot2^27^2,1}(12u+6)$: the root number of an integer fibre is constant by Lemma \ref{mainthmL}, and it can be checked to be $1$. However, if $u=1$ then $W(\mathcal{F}_{-3\cdot2^27^4}(7t+7))=-1$. Thus $\mathcal{L}_{7,-3\cdot2^27^2,1}(12u+6)$ is a subfamily of an $\mathcal{F}$-family on which the root number is non-constant and is therefore not contained in Theorem \ref{thm: basechange}. 
\end{exam}

For a complete description of $\mathcal{L}$-families that have constant root number over $\Z$, see our upcoming paper. For now we state some sufficient conditions.

\begin{lemma}\label{mainthmL}
Let $a,b,w,r,v\in \Z$ with $w,r,v$ non-zero. Suppose they satisfy the following three conditions:
    \begin{enumerate}
    \item for $p\geq 5$ such that $p\mid r$, $\nu_p(v)=\nu_p(a)=\nu_p(b)=0$, $\nu_p(w)$ is odd, and $(\frac{-v}{p})=-1$, 
    \item $\nu_3(a)>0$, $\nu_3(b)>0$ and $\nu_3(v)=0$,
    \item $\nu_2(a)>1$, $\nu_2(b)>0$ and $\nu_2(v)=0$. 
    \end{enumerate} 
Then the family $\mathcal{L}_{w,-3r^2,v}(au+b)$ has constant root number over $\Z$.
\end{lemma}

\begin{rem}
The family $\mathcal{V}_v$ in Example \ref{ex: basechange} is covered by both Corollary \ref{cor: main_to_L} and Theorem \ref{mainthmL} if $3\nmid v$, but only by Corollary \ref{cor: main_to_L} if $3\mid v$. These two results about $\mathcal{L}$-families are not equivalent nor does one imply the other.
\end{rem}

\begin{proof}

Let $p\geq5$. We refer to Proposition \ref{prop: w_p^*(t)} for computing the value of the function $w_p^*(T(u))$, using $S=-3r^2w^2$ rather than $s$ and $T(u)$ rather than $t$. 

By Condition 1 in Lemma \ref{mainthmL}, we have $p\mid r$. It will be useful to use the notation $T(u)=wR(u)$ where $R(u)=(au+b)^2+v$, so that $\nu_p\left(T(u)\right)=\nu_p(w)+\nu_p\left(R(u)\right)$ and $T(u)_p=w_pR(u)_p$.

Suppose $(\frac{-v_{p}}{p})=-1$, i.e. $-v_{p}$ is not a square modulo $p$. Then given that $\nu_p(v)<\max(\nu_p(a),\nu_p(b))$, we have $p\nmid (au+b)^2+v$ for any $u$. This means that $\nu_p(T(u))=\nu_p(w)$. Indeed, there is no $u$ such that $p\mid R(u)$. 
Suppose $\nu_p(w)$ is odd, then $\nu_p(T(u))=\nu_p(w)\equiv1\pmod2$ and remember that $2\nu_p(T(u))-\nu_p(S)=-\nu_p(r)<0$. Then by Proposition \ref{prop: w_p^*(t)} for any $u$ 
we have $w_p^*(T(u))=(\frac{-1}{p})$. 

For $p=3$, let $\nu_3(b)\not=0$, $\nu_3(a)\not=0$ and $\nu_3(v)=0$. Then $\nu_3(R)=0$ and $R_3\equiv v\pmod3$. Thus the function $w_3^*(T)$ is constant.

For $p=2$, first let $\nu_2(a)>0$, $\nu_2(b)>0$ and $\nu_2(v)=0$. Then $\nu_2(R)=0$ and so $\nu_2(S)-2\nu_2(T)=2\nu_2(r) \geq 0$ is constant. Also we have
    \begin{align*}
        R_2 \equiv \begin{cases}
        v \pmod 8, &\nu_2(a)>1,\nu_2(b)>1\\
        v + 4u^2 \pmod 8, &\nu_2(a)=1,\nu_2(b)>1\\
        v + 4 \pmod 8, &\nu_2(a)>1,\nu_2(b)=1\\
        v + 4u^2 + 4\pmod 8, &\nu_2(a)=1,\nu_2(b)=1.
        \end{cases}
    \end{align*}
    Then for $\nu_2(a)>1$ we see that $R_2$ is independent of $u$, and hence $w^*(T)$ is constant.  \qedhere
\end{proof}

\section{Rank jump}\label{sec: 4}

Let $\E$ be a family of elliptic curves. We define the \emph{generic rank} $r(\E)$ to be the rank of $\E$ as an elliptic curve over $\Q(t)$. We generally compare the generic rank and the rank of the fibres using the following theorem:

\begin{theorem}[Silverman's Specialization Theorem \cite{Silv}]\label{thm: silv}
One has $r(\E)\leq r(\E(t))$ for all but finitely many $t\in\Q$.
\end{theorem}

We say $\E$ has a \emph{rank jump} (or a \emph{rank elevation}) on the integer fibres if $r(\E)< r(\E(t))$ for all but finitely many $t\in\Z$. We state a result from Proposition 5 of \cite[p.~15]{BDD}.
\begin{prop}\label{prop:BDD_generic_rank} The generic rank of $\mathcal{F}_s$ is 1 if and only if $s=-12k^4$ for $k\in \mathbb{N}$ (where otherwise the generic rank is 0).
\end{prop}

Recall that the root number of an elliptic curve over $\Q$ is conjecturally equal to the parity of the rank (Conjecture \ref{parityconjecture}). Combining this conjecture, Theorem \ref{thm: silv} and Proposition \ref{prop:BDD_generic_rank}, we deduce that if $s=-12k^4$ and $W(\mathcal{F}_s(au+b))=1$ for all $u\in\Z$, then there is a rank jump i.e. all but finitely many non-singular integer fibres $\mathcal{F}_{s}(au+b)$ have rank 2 or more. Likewise, if $s\not=-12k^4$ and $W(\mathcal{F}_s(au+b))=-1$ for all $u\in\Z$, then there is a rank jump and all but finitely many non-singular integer fibres have rank 1 or more.


We investigate the case $s=-12k^4$ in order to find examples of integer fibres with rank at least 2.
With the specifications found in Theorem \ref{mainthm} for $s, a, b$, it suffices to narrow down these conditions so that $W(\mathcal{F}_s(t))=1$ for all $t\in a\Z+b$. Recall from Theorem \ref{thm: indep} that $W(\mathcal{F}_s(t))=-w_2^*(t)w_3^*(t)\prod_{p\neq 2,3} w_p^*(t)$. We showed in the proof of Theorem \ref{thm: condition sab} that $w_p^*(t)=1$ for $p\nmid s$. Hence writing the prime decomposition  $s=2^{\alpha_2}3^{\alpha_3}q_1^{\alpha_{q_1}}\cdots q_n^{\alpha_{q_n}}$, we may express the root number as
    \begin{equation}\label{eq: 9}
        W(\mathcal{F}_s(t))=-w_2^*(t)w_3^*(t)\prod_{i=1}^n w_{q_i}^*(t)
    \end{equation}
    where the local functions are independent from one another. So in order for $W(\mathcal{F}_s(t))=1$, we need a subset $P\subseteq \{2,3,q_1,...,q_n\}$ of odd cardinality such that for $p\in P$,  $w_p^*(t)=-1$ for all $t\in a\Z+b$ while for $p\notin P$, $w_p^*(t)=1$ for all $t\in a\Z+b$. To do so, we find conditions on $s,a,b,$ under which $w_p^*(t)=1$ for each $p\mid s$ and likewise for $w_p^*(t)=-1$. These are described in Propositions \ref{prop: s=-12k^4 p5}, \ref{prop: s=-12k^4 p3}, and \ref{prop: s=-12k^4 p2}. In these results we assume $s=-12k^4$, which we note is of the form $s=-3r^2$ for some $r\in \Z$. Further, $\nu_p(s)\equiv 0 \pmod 4$ for $p\geq 5$, $\nu_3(s)\equiv 1 \pmod 4$ and $\nu_2(s)\equiv 2 \pmod 4$.

\begin{prop} \label{prop: s=-12k^4 p5} Let $p\geq 5 $ with $p\mid k$. Consider the following conditions on $s,a,b$:
    \begin{enumerate}
        \item $\nu_p(b) < \nu_p(a)$ and $\nu_p(s)\leq 2\nu_p(b)$,
        \item $\nu_p(b) < \nu_p(a)$ and $\nu_p(s)> 2\nu_p(b)$ and
            \begin{enumerate}
                \item $\nu_p(b)\equiv 0\pmod 2$ and $(\frac{b_p}{p})=-1$, or
                \item $\nu_p(b)\equiv 0\pmod 2$ and $(\frac{b_p}{p})=1$, or
                \item$\nu_p(b)\equiv 1\pmod 2$
            \end{enumerate}
        \item $\nu_p(a) \leq \nu_p(b)$ and $\nu_p(s)\leq 2\nu_p(b)$.
    \end{enumerate}
    
If $p\equiv 1 \pmod 4$, then $w_p^*(au+b)=1$ for all $u\in\Z$ if and only if one of the conditions 1, 2(a), 2(c) or 3 holds.
	 and $w_p^*(au+b)=-1$ for all $u\in\Z$ if and only if condition 2(b) holds.
	If $p\equiv 3\pmod 4$, then $w_p^*(au+b)=1$ for all $u\in\Z$ if and only if one of the conditions 1, 2(b) or 3 holds,
	and $w_p^*(au+b)=-1$ for all $u\in\Z$ if and only if one of the conditions 2(a) or 2(c) holds.
\end{prop}			

\begin{prop}\label{prop: s=-12k^4 p3}
    We have $w_3^*(au+b)=1$ for all $u\in\Z$ if and only if one of the following holds:
	\begin{enumerate}
		\item $\nu_3(b)<\nu_3(a)$ and
		    \begin{enumerate}
		        \item $\nu_3(s)\leq 2\nu_3(b)+3$, or
		        \item $\nu_3(s)=2\nu_3(b)+3+4\ell$, $\ell\geq 1$ and $b_3\equiv 2 \pmod 3$
		    \end{enumerate}
		\item $\nu_3(a)\leq \nu_3(b)$ and $\nu_3(s)\leq 2\nu_3(a)+3$
	\end{enumerate}
	and $w_3^*(au+b)=-1$ for all $u\in \Z$ if and only if $\nu_3(b)<\nu_3(a)$ and one of the following:
	    \begin{enumerate}\setcounter{enumi}{2}
	        \item 
	        $\nu_3(s)=2\nu_3(b)+1+4\ell$, $\ell\geq 1$,
	        \item  $\nu_3(s)=2\nu_3(b)+3+4\ell$, $\ell\geq 1$, and $b_3\equiv 1 \pmod 3$.
	    \end{enumerate}
	
	\end{prop}
	
\begin{prop}\label{prop: s=-12k^4 p2}
We have $w_2^*(au+b)=1$ for all $u\in \Z$ if and only if one of the following holds:
		
	\begin{enumerate}
		\item $\nu_2(b)<\nu_2(a)-2$ and
		    \begin{enumerate}
		        \item $\nu_2(s)=2\nu_2(b)-4$, or
		        \item $\nu_2(s) = 2\nu_2(b)$ and
		            \begin{enumerate}
		                \item $b_2 \equiv 1 \pmod 4$, or
		                \item $b_2 \equiv 3 \pmod 8$ and $(b_2^2-s)_2 \equiv 3 \pmod 4$
		            \end{enumerate}
		        \item $\nu_2(s)=2\nu_2(b)+2$ and $b_2\equiv 1, 3, 7 \pmod 8$, or
		        \item $\nu_2(s)=2\nu_2(b)+4\ell$, $\ell\geq 1$ and $b_2\equiv 3 \pmod 4$, or
		       	\item $\nu_2(s)=2\nu_2(b)+6$ and $b_2\equiv 3 \pmod 4$, or
		        \item $\nu_2(s)=2\nu_2(b)+2+4\ell$, $\ell \geq 1$ and $b_2\equiv 7 \pmod 8$
		    \end{enumerate}
		\item $\nu_2(b)=\nu_2(a)-2$ and
		    \begin{enumerate}
		        \item $\nu_2(s)\leq 2\nu_2(b)-4$, or
                \item $\nu_2(s) = 2\nu_2(b)$ and
		            \begin{enumerate}
		                \item $b_2 \equiv 1 \pmod 4$, or
		                \item $b_2 \equiv 3 \pmod 8$ and $(b_2^2-s)_2 \equiv 3 \pmod 4$
		            \end{enumerate}
        		\item $\nu_2(s)=2\nu_2(b)+2$ and $b_2\equiv 3 \pmod 4$, or
        		\item $\nu_2(s)=2\nu_2(b)+4\ell$, $\ell\geq 1$ and $b_2\equiv 3 \pmod 4$, or
        		\item $\nu_2(s)=2\nu_2(b)+6$ and $b_2\equiv 3 \pmod 4$	, or
		    \end{enumerate}
        \item $\nu_2(b)=\nu_2(a)-1$ and $\nu_2(s)\leq 2\nu_2(b)-4$
		\item $\nu_2(a)\leq \nu_2(b)$ and $\nu_2(s)\leq 2\nu_2(a)-4$	
	\end{enumerate}	
and $w_2^*(au+b)=-1$ for all $u\in \Z$ if and only if one of the following holds:
		\begin{enumerate}\setcounter{enumi}{4}
		\item $\nu_2(b)<\nu_2(a)-2$ and
		    \begin{enumerate}
		        \item $\nu_2(s)<2\nu_2(b)-4$, or
		        \item $\nu_2(s)=2\nu_2(b)-2$, or
		  		\item $\nu_2(s) = 2\nu_2(b)$ and
		            \begin{enumerate}
		                \item $b_2 \equiv 3 \pmod 8$ and $(b_2^2-s)_2 \equiv 1 \pmod 4$
		                \item $b_2 \equiv 7 \pmod 8$
		            \end{enumerate}
        		\item $\nu_2(s)=2\nu_2(b)+2$ and $b_2\equiv 5 \pmod 8$, or
        		\item $\nu_2(s)=2\nu_2(b)+4\ell$, $\ell\geq 1$ and $b_2\equiv 1 \pmod 4$, or
        		\item $\nu_2(s)=2\nu_2(b)+6$ and $b_2\equiv 1 \pmod 4$, or
        		\item $\nu_2(s)=2\nu_2(b)+2+4\ell$, $\ell \geq 1$ and $b_2\equiv 1, 3, 5 \pmod 8$
		    \end{enumerate}
		\item $\nu_2(b)=\nu_2(a)-2$, and
		    \begin{enumerate}
		        \item $\nu_2(s)=2\nu_2(b)-2$ , or
		  		\item $\nu_2(s) = 2\nu_2(b)$ and
		            \begin{enumerate}
		                \item $b_2 \equiv 3 \pmod 8$ and $(b_2^2-s)_2 \equiv 1 \pmod 4$
		                \item $b_2 \equiv 7 \pmod 8$
		            \end{enumerate}
		        \item $\nu_2(s)=2\nu_2(b)+4\ell$, $\ell\geq 1$ and $b_2\equiv 1 \pmod 4$, or
		        \item $\nu_2(s)=2\nu_2(b)+6$ and $b_2\equiv 1 \pmod 4$, or
		        \item $\nu_2(s)=2\nu_2(b)+2+4\ell$, $\ell \geq 1$ and $b_2\equiv 1 \pmod 4$
		    \end{enumerate}
		\item $\nu_2(b)=\nu_2(a)-1$ and $\nu_2(s) = 2\nu_2(b)-2$.
		\end{enumerate}
		
\end{prop}

We use Theorem \ref{thm: condition sab} (Theorem \ref{mainthm}) continuously through the proofs of these Propositions.
	\begin{proof}[Proof of Proposition \ref{prop: s=-12k^4 p5}] 
 For $p\geq 5$, we have $\nu_p(s)\equiv 0 \pmod 4$. 
		For $\nu_p(b)<\nu_p(a)$, we reorganize parts of Table \ref{table: w_p^*(t)} to get
		\begin{center}
		\begin{tabular}{l l l l r}
		\hline
		 $p$&$\nu_p(s)-2\nu_p(b)$	&$\nu_p(b)$	&$(\tfrac{b_p}{p})$&$w_p^*(t)$	\\
		\hline
		$1\pmod 4$   &$>0$ &$0\pmod{2}$&$-1$		&$1$\\
			&	&	&$\phantom{-}1$	&$-1$	\\
			&	        &$1\pmod 2$		&			&1\\
	       &0			&		&			&1	\\
			&$<0$	&   	&			&1\\
		$3\pmod 4$	&$>0$	&$0\pmod{2}$&$\phantom{-}1$	&1\\
				&	&	&$-1$&$-1$\\
				&		&$1\pmod 2$	&	&$-1$			\\
				&0			&	&			&1\\
				&$<0$	&				&	&1\\
				\hline
		\end{tabular}
		\end{center}
		from which the result follows. For $\nu_p(s)/2 \leq \nu_p(a)\leq \nu_p(b)$, $w_2^*(t)=1$. 
%
		\end{proof}
		
		\begin{proof}[Proof of Proposition \ref{prop: s=-12k^4 p3}]
	For $p=3$, we have $\nu_3(s)\equiv 1 \pmod 4$ and $s_3=-4k_3^4\equiv 2 \pmod 3$. We refer to Section \ref{sec:mainthm_pf_p=3}, in particular Table \ref{table:p=3_rank}.
	\end{proof}
	
	\begin{proof}[Proof of Proposition \ref{prop: s=-12k^4 p2}]
	For $p=2$, we have $\nu_2(s)\equiv 2 \pmod 4$ and $s_2\equiv 1 \pmod 4$. For $2\nu_2(b)\neq \nu_2(s)$, we refer to the proof Theorem \ref{thm: condition sab} and Table \ref{table: w_2^*(t),v_2(s)2mod4}. For $2\nu_2(b)= \nu_2(s)$, note that since $\nu_2(s)\equiv 2 \pmod 4$, we have that $\nu_2(b)$ is odd. Then for $\nu_2(b)\leq \nu_2(a)-2$, we saw from Theorem \ref{thm: condition sab} that $\nu_2(t^2-s)-2\nu_2(t)=2$ and $(t^2-s)_2 \equiv (b_2^2-s_2)_2\pmod 4$ and so from Table \ref{table: w_2(t) for 2v_2(t)=v(s),v_2(t)odd} we get
	            \begin{align*}
	                w_2^*(t)=\begin{cases}
                    	1,  \ &b_2\equiv 1\pmod{4}\\
                    	& b_2\equiv 3\pmod{8}\text{ and }(b_2^2-s)_2\equiv3\pmod{4}\\
                    	-1, \ &\text{otherwise}
                    	\end{cases}
                \end{align*}
    and this concludes the proof.
	\end{proof}

We can be slightly more specific in the hypothesis on $s$ so that it is easier to produce examples. Recall again Washington's example which has the form $\mathcal{F}_{-2^23^5}(12u+18)$; it is a subfamily of a certain $\mathcal{F}_s$ where $s=-12k^4$ with $k=3$. This family has generic rank $1$, but the root number is $-1$ for all fibres, which means that there is no rank jump. One way to generalize Washington's example but instead with families in which there is a rank jump is to study the case $k$ is a prime. So let $k=q$, an odd prime number. Then the Propositions \ref{prop: s=-12k^4 p5}, \ref{prop: s=-12k^4 p3} and \ref{prop: s=-12k^4 p2} reduce to the following corollary: 

\begin{cor}\label{prop: k=p} Let $s=-12q^4$ for some prime number $q\geq 5$.	\begin{enumerate}
	\item 
		If $q\equiv 1 \pmod 4$, then $w_q^*(t)=1$ if and only if one of the following
			\begin{enumerate}		
			\item 
			$\nu_q(a) \leq \nu_q(b)$ and $\nu_q(b)\geq 2$
			\item $1\leq \nu_q(b) < \nu_q(a)$
			\item $0= \nu_q(b) < \nu_q(a)$ and $(\frac{b}{q})=-1$
			\end{enumerate}	
		and $w_q^*(t)=-1$ if and only if $0=\nu_q(b)<\nu_q(a)$ and $(\frac{b}{q})=1$.
		
		If $q\equiv 3\pmod 4$, then $w_q^*(t)=1$ if and only if one of the following
		\begin{enumerate}
		\item 
		$\nu_q(a) \leq \nu_q(b)$ and $\nu_q(b)\geq 2$
		\item $2\leq \nu_q(b) < \nu_q(a)$
		\item $0= \nu_q(b) < \nu_q(a)$ and $(\frac{b}{q})=1$
	and $w_q^*(t)=-1$ if and only if
		\begin{enumerate}
		 \item $0=\nu_q(b)<\nu_q(a)$ and $(\frac{b}{q})=-1$
		 \item $1=\nu_q(b)<\nu_q(a)$.
		\end{enumerate}	
\end{enumerate}		
	\item $w_3^*(t)=1$ for any case.
	\item $w_2^*(t)=1$ if and only if one of the following
		\begin{enumerate}
		\item $\nu_2(b)\geq 3$ and $\nu_2(a)=\nu_2(b)+2$ 
		\item $\nu_2(b)\geq 3$ and $\nu_2(a)=\nu_2(b)+1$ 
		\item $\nu_2(b)=3$ and $\nu_2(a)>\nu_2(b)+2$
		\item $\nu_2(b)\geq1$ and $\nu_2(a)=\nu_2(b)+2$
		\item 
		$\nu_2(b)=1$ and $\nu_2(a)\geq 3$ and
		    \begin{enumerate}
		        \item $b_2\equiv 1 \pmod 4$, or
		        \item $b_2\equiv 3 \pmod 8$ and $(b_2^2-s_2)_2 \equiv 3 \pmod 4$
		    \end{enumerate}
		\item $\nu_2(b)=0$ and $\nu_2(a)>2$ and $b\equiv 1, 3, 7 \pmod 8$
		\item $\nu_2(b)=0$ and $\nu_2(a)=2$ and $b \equiv 3\pmod 4$
		\item $3\leq \nu_2(a)\leq \nu_2(b)$
		\end{enumerate}
	and $w_2^*(t)=-1$ if and only if one of the following
		\begin{enumerate}
		\item $3<\nu_2(b)<\nu_2(a)-2$,
		\item $\nu_2(b)=2$ and $\nu_2(a)\geq 3$,
		\item $\nu_2(b)=1$ and $\nu_2(a)\geq3$ and $b_2\equiv3\pmod 8$ and $(b_2^2-s)_2\equiv1\pmod4$,
		\item $\nu_2(b)=1$ and $\nu_2(a)\geq3$ and $b_2\equiv7\pmod 8$,
		\item $\nu_2(b)=0$ and $\nu_2(a)>2$ and $b_2\equiv 5 \pmod 8$.
		\end{enumerate}
	\end{enumerate}
\end{cor}

\begin{proof} In this case, $s=-3 \cdot 2^2 q^4$, so $\nu_q(s)=4, \nu_3(s)=1$ and  $\nu_2(s)=2$. 
	For $p \geq 5$, if $p\neq q$, then $w_p^*(t)=1$. Else if $p=q$, then we have $\nu_p(s)/2=2$. For $\nu_p(b)<\nu_p(s)/2$, we get $\nu_p(b)=0$ or $1$. Hence the conditions $\nu_p(b)\equiv 0 \pmod 2$ and $(\frac{b_p}{p})=-1$ forces $\nu_p(b)$ to be 0 and $b=b_p$. Similarly, the condition $\nu_p(b)\equiv 1 \pmod 2$ forces $\nu_p(b)$ to be 1.
	
	For $p=3$, since $\nu_3(s)=1$, $\nu_3(s)-2\nu_3(b)\leq 1$ and so we cannot have the cases $\nu_3(s)=2\nu_3(b)+1+4\ell$ nor $\nu_3(s)=2\nu_3(b)+3+4\ell$ for some $\ell\geq 1$ and hence there is no case such that $w_3^*(t)=-1$. Then the conditions for $w_3^*(t)=1$ become $0\leq \nu_3(b) < \nu_3(a)$ and $0\leq \nu_3(a)\leq \nu_3(b)$ which is the same as no conditions.
	
	For $p=2$, since $\nu_2(s)=2$, $\nu_2(s)-2\nu_2(b)\leq 2$ i.e. $(\nu_2(s)-2)/2 \leq \nu_2(b)$. This eliminates some cases from Proposition \ref{prop: s=-12k^4 p2}. The remaining we obtain by evaluating $\nu_2(s)=2$.
	\end{proof}

\begin{exam} Let $s=-12\cdot5^4$,
$a=2^{3+\alpha_2} 3^{\alpha_3} 5^{2+ \alpha_5} a_0$, $b=2^2 3^{\beta_3} 5 b_0$
	where $a_0,b_0$ are coprime to $s$ and $\alpha_2, \alpha_3, \alpha_5, \beta_3$ are non-zero integers. Then by Corollary \ref{prop: k=p}, $w^*_5(au+b)=1, w_3^*(au+b)=1$ and $w_2^*(au+b)=-1$ and  so by \eqref{eq: 9} we have $W(\mathcal{F}_s(au+b))=1$. 
	According to the Parity Conjecture combined with Silverman's Specialization Theorem, the rank is at least $2$ for all but finitely many integer fibres of $\mathcal{F}_{s}(au+b)$.
	\end{exam}

\newpage
\appendix

\section{Formulas} \label{sec: app} \addtocontents{toc}{\setcounter{tocdepth}{2}}

	\subsection{Formulas for $w_p^*(t)$ on $\mathcal{F}_s(t)$}\label{modifiedlocalrootnumber}
We report here the formulas of the functions $w_p^*$ on the fibres of an elliptic surface $\mathcal{F}_s(t)$. 
\begin{prop}\label{prop: w_p^*(t)} For $p\geq 5$,
	\begin{table}[H] 
	\caption{} \label{table: w_p^*(t)}
	\centering
	\begin{tabular}{l l l l r}
	\hline
	$\nu_p(s$)	&$2\nu_p(t)-\nu_p(s)$ 	&$\nu_p(t)$&$\nu_p(t^2-s)$	&$w_p^*(t)$\\
	\hline
	$1\pmod{2}$	&$<0$	&$0\pmod{2}$	&&$-(\frac{3t_p}{p})$\\
						&			&$1\pmod{2}$	&&$(\frac{-1}{p})$\\
						&$>0$	&						&&$(\frac{2}{p})$\\
	$0\pmod{4}$&$<0$	&$0\pmod{2}$	&&$-(\frac{3t_p}{p})$\\
						&			&$1\pmod{2}$	&&$(\frac{-1}{p})$\\
						&$>0$	&						&&1\\
						&$0$	&$0\pmod{6}$	&$2,4 \pmod{6}$	&$(\frac{-3}{p})$\\
						&			&						&otherwise	&1\\
						&   	&$2\pmod{6}$	&$0,2 \pmod{6}$	&$(\frac{-3}{p})$\\
						&			&						&otherwise	&1\\
						&   	&$4\pmod{6}$	&$0,4 \pmod{6}$	&$(\frac{-3}{p})$\\
						&			&						&otherwise	&1\\			
	$2\pmod{4}$&$<0$	&$0\pmod{2}$	&&$-(\frac{3t_p}{p})$\\
						&			&$1\pmod{2}$	&&$(\frac{-1}{p})$\\
						&$>0$	&						&&$(\frac{-1}{p})$\\
						&$0$	&$1\pmod{6}$	&$1,3 \pmod{6}$	&$(\frac{3}{p})$\\
						&			&						&otherwise	&$(\frac{-1}{p})$\\
						&   	&$3\pmod{6}$	&$1,5 \pmod{6}$	&$(\frac{3}{p})$\\
						&			&						&otherwise	&$(\frac{-1}{p})$\\
						&   	&$5\pmod{6}$	&$3,5 \pmod{6}$	&$(\frac{3}{p})$\\
						&			&						&otherwise	&$(\frac{-1}{p})$\\		
	\hline
	\end{tabular}
	\end{table}
\end{prop}
	
	\subsection{Formulas for $w_3^*(t)$ on $\mathcal{F}_s(t)$}
	
\begin{prop}\label{prop: w_3^*(t)}  For $p=3$,
	\begin{itemize}
	\item    For $\nu_3(s)\equiv 1\pmod{2}$,
		\begin{table}[H]
		\caption{} \label{table: w_3^*(t),v_3(s)odd}
		\centering
		\begin{tabular}{l l r}
		\hline
		$\nu_3(s)$	&$\nu_3(s)-2\nu_3(t)$	&$w_3^*(t)$\\
		\hline
		$1 \pmod{4}$	&$<-1$						&$1$\\
			&$-1$						&$-(\frac{s_3}{3})$\\
			&$1, 3$						&$1$\\
			&$1\pmod{4}, >1$		&$-1$\\
			&$3\pmod{4}, >3$	&$-(\frac{t_3}{3})$\\
		$3 \pmod{4}$	&$1$							&$(\frac{s_3}{3})$\\
			&$1\pmod{4}, >1$		&$-(\frac{t_3}{3})$\\
			&otherwise				&$-1$\\
		\hline
		\end{tabular}
		\end{table}
		
	\item    For $\nu_3(s)\equiv 0\pmod{4}$,
		\begin{itemize}
		\item For $\nu_3(s)\neq 2\nu_3(t)$,
			\begin{table}[H]
			\caption{}
			\centering
			\begin{tabular}{l l r}
			\hline
			$\nu_3(s)-2\nu_3(t)$	&$\nu_3(t)$	&$w_3^*(t)$\\
			\hline
			$<-2$	&						&1\\
			$-2$		&						&$(\frac{t_3}{3})$\\
			$>0$	&$0\pmod{2}$	&$-1$\\
						&$1\pmod{2}$	&$-(\frac{t_3}{3})$\\
						\hline
			\end{tabular}
			\end{table}
			
		\item For $\nu_3(s)=2\nu_3(t)$, $w_3^*(t)=1$ if and only if 
			\begin{itemize}
			\item    $\nu_3(t^2-s)-2\nu_3(t)=0$ and $s_3\equiv 2 \pmod{3}$ and $s_3t_3\not\equiv 2, 4 \pmod{9}$
			\item    or one of the following
				\begin{table}[H]
				\caption{}
				\centering
				\begin{tabular}{l l}
				\hline
				$\nu_3(t^2-s)-2\nu_3(t)$	& $t_3(t^2-s)_3$\\
				\hline
				$0\pmod{6}, \ >0$	&$\not\equiv 7, 8 \pmod{9}$	\\
				$1 \pmod{6}$	&$2\pmod{3}$				\\
				$2 \pmod{6}$	&$1\pmod{3}$				\\
				$3\pmod{6}$	&$ 1, 2 \pmod{9}$	\\
				$4 \pmod{6}$	&$2\pmod{3}$				\\
				$5 \pmod{6}$	&$1\pmod{3}$				\\
				\hline
				\end{tabular}	
				\end{table}
				\end{itemize}	
		\end{itemize}
		
	\item    For $\nu_3(s)\equiv 2\pmod{4}$,
		\begin{itemize}
		\item    For $\nu_3(s)\neq 2\nu_3(t)$,
			\begin{table}[H]
			\caption{}
			\centering
			\begin{tabular}{l l r}
			\hline
			$\nu_3(s)-2\nu_3(t)$	&$t_3$	&$w_3^*(t)$\\
			\hline
			$<-2$	&												&$1$\\
			$-2$		&$s_3\pmod{3}$		&1\\
						&$-s_3\pmod{3}$	&$-1$\\
			2			&$s_3\pmod{3}$		&1\\
						&$-s_3\pmod{3}$	&$-1$\\
			$0\pmod{4}, >0$	&								&$-(\frac{t_3}{3})$\\
			$2\pmod{4}, >2$	&								&$-1$\\			\hline
			\end{tabular}
			\end{table}
			
		\item    For $\nu_3(s)=2\nu_3(t)$, $w_3^*(t)=1$ if and only if
			\begin{itemize}
			\item    $\nu_3(t^2-s)-2\nu_3(t)=0$ and $s_3\equiv 2 \pmod{3}$ and $s_3t_3\not\equiv 2, 4 \pmod{9}$
			\item    or one of the following
				\begin{table}[H]
				\caption{}
				\centering
				\begin{tabular}{l l}
				\hline
				$\nu_3(t^2-s)-2\nu_3(t)$	& $t_3(t^2-s)_3$	\\
				\hline
				$0\pmod{6}, \ >0$	&$\not\equiv 1,2 \pmod{9}$		\\
				$1 \pmod{6}$	&$1\pmod{3}$				\\
				$2 \pmod{6}$	&$2\pmod{3}$				\\
				$3\pmod{6}$	&$ 7,8 \pmod{9}$	\\
				$4 \pmod{6}$	&$1\pmod{3}$				\\
				$5 \pmod{6}$	&$2\pmod{3}$				\\
				\hline
				\end{tabular}	
				\end{table}
			\end{itemize}
		\end{itemize}
	\end{itemize}
\end{prop}
\subsection{Formulas for $w_2^*(t)$ on $\mathcal{F}_s(t)$}

\begin{prop}\label{prop: w_2^*(t)}  For $p=2$,
	\begin{itemize}
	\item    For $\nu_2(s)\equiv 0 \pmod{4}$,
	    \begin{itemize}
	        \item For $\nu_2(s) \neq 2\nu_2(t) $,
	        
			\begin{table}[H]
			\caption{} \label{table: w_2^*(t),v_2(s)0mod4}
			\centering
			\begin{tabular}{l l l r}
			\hline
			$\nu_2(s)-2\nu_2(t)$	&$s_2$	&$t_2$ &$w_2^*(t)$\\
			\hline
			$< -4$	&$1,3,7,13$ or $15\pmod{16}$&						&$-1$\\
					&$5,9$ or $11\pmod{16}$					&						&$\phantom{-}1$\\
			$-4$	&$3,5,7,9,11$ or $15\pmod{16}$		&	&$-1$\\
					&$1$ or $13\pmod{16}$			&	&$\phantom{-}1$\\
			$-2$	&$3\pmod{4}$				&						&$\phantom{-}1$\\
					&1 or $13\pmod{16}$	&$1\pmod{4}$	&$\phantom{-}1$\\
					&5 or $9\pmod{16}$	&$3\pmod{4}$	&$\phantom{-}1$\\
					&otherwise					&			&$-1$\\
			2		&					&$s_2\pmod{4}$	&$\phantom{-}1$\\
					&					&$-s_2\pmod 4$	&$-1$\\
			$2\pmod{4}, >2$	&&$3\pmod{4}$	&$\phantom{-}1$\\
										&&$1\pmod{4}$	&$-1$\\
			4		&						&1 or $5 \pmod{8}$ 			&$\phantom{-}1$\\
					&$1\pmod{4}$	&$3 \pmod{8}$ 				&$\phantom{-}1$\\
					&$3\pmod{4}$	&$7 \pmod{8}$				 	&$\phantom{-}1$\\
					&otherwise		&										&$-1$\\
			$0\pmod{4}, >4$	&	&$7\pmod{8}$	 				&$\phantom{-}1$\\
										&	&otherwise						&$-1$\\
			\hline
			\end{tabular}
			\end{table}
		\item For $\nu_2(s)=2\nu_2(t)$, the conditions are listed in Table \ref{table: w_2(t) for 2v_2(t)=v(s),v_2(t)even}.			
			
		\begin{table}[H] 
		\caption{}
		\label{table: w_2(t) for 2v_2(t)=v(s),v_2(t)even}
		\centering
		\begin{tabular}{l l r} 
		\hline
		$\nu_2(t^2-s)-2\nu_2(t)$ &conditions&$w_2^*(t)$\\
		\hline
		$1$&$t_2\equiv1\pmod{8}$, $(t^2-s)_2\equiv1,7\pmod{8}$&$1$\\
		&$t_2\equiv3\pmod{8}$, $(t^2-s)_2\equiv5,7\pmod{8}$&\\
		&$t_2\equiv5\pmod{8}$, $(t^2-s)_2\equiv3,5\pmod{8}$&\\
		&$t_2\equiv7\pmod{8}$, $(t^2-s)_2\equiv1,3\pmod{8}$&\\
		&otherwise&$-1$\\
		2&$t_2(t^2-s)_2\equiv3\pmod{4}$&1\\
		&otherwise&$-1$\\
		3&$t_2\equiv1\pmod{8}$, $(t^2-s)_2\equiv3,5\pmod{8}$&$1$\\
		&$t_2\equiv3\pmod{8}$, $(t^2-s)_2\equiv1,3\pmod{8}$&\\
		&$t_2\equiv5\pmod{8}$, $(t^2-s)_2\equiv1,7\pmod{8}$&\\
		&$t_2\equiv7\pmod{8}$, $(t^2-s)_2\equiv5,7\pmod{8}$&\\
		&otherwise&$-1$\\
		5&$(t^2-s)_2\equiv1\pmod{8}$, $t_2\equiv1,3,7\pmod{8}$&$1$\\
&$(t^2-s)_2\equiv3\pmod{8}$, $t_2\equiv1,3,5\pmod{8}$&\\
&$(t^2-s)_2\equiv5\pmod{8}$, $t_2\equiv1,3,5\pmod{8}$&\\
&$(t^2-s)_2\equiv7\pmod{8}$, $t_2\equiv1,5,7\pmod{8}$&\\
&otherwise&\\
		otherwise&&$t_2\pmod{4}$\\
\hline
%
		\end{tabular} 
		\end{table}
			    \end{itemize}
	\item    For $\nu_2(s)\equiv 1 \pmod{4}$,
			\begin{table}[H]
			\caption{}
			\centering
			\begin{tabular}{l l l r}
			\hline
			$\nu_2(s)-2\nu_2(t)$	&$s_2$	&$t_2$ &$w_2^*(t)$\\
			\hline
			$\leq -2$	&$3,5\pmod{8}$	&							&$-1$\\
					&$1,7\pmod{8}$			&							&$\phantom{-}1$\\
			$-1$	&$1,7 \pmod{8}$	&$1\pmod{4}$		&$\phantom{-}1$\\
			&$3,5\pmod{8}$&$3\pmod{4}$&$\phantom{-}1$\\
					&otherwise				&							&$-1$\\
			1		&$1\pmod{4}$			&1 or $7\pmod{8}$	&$\phantom{-}1$\\
					&$3\pmod{4}$			&1 or $3\pmod{8}$	&$\phantom{-}1$\\
					&otherwise				&								&$-1$\\
			5		&								&1, 5, or 7	$\pmod{8}$	&$\phantom{-}1$\\
					&								&otherwise				&$-1$\\
		$1\pmod{4}, >5$	&				&$7\pmod{8}$			&$\phantom{-}1$\\
					&								&otherwise				&$-1$\\
			3		&$3\pmod 4$			&							&$\phantom{-}1$\\
					&	$1\pmod 4$			&							&$-1$\\
		$3\pmod{4}, >3$	&	&$3\pmod 4$					&$\phantom{-}1$\\
					&					&$1\pmod 4$					&$-1$\\
			\hline
			\end{tabular}
			\end{table}
		
	\item    For $\nu_2(s)\equiv 2 \pmod{4}$,
	    \begin{itemize}
	        \item For $\nu_2(s)\neq 2 \nu_2(t)$,
			\begin{table}[H]
			\caption{} \label{table: w_2^*(t),v_2(s)2mod4}
			\centering
			\begin{tabular}{l l l r}
			\hline
			$\nu_2(s)-2\nu_2(t)$	&$s_2$	&$t_2$	&$w_2^*(t)$\\
			\hline
			$< -4$	&$1,3,5,9,13$ or $15\pmod{16}$&									&$\phantom{-}1$\\
					&$7,11\pmod{16}$				&										&$-1$\\ 
			$-4$	&$1,5,7,9,11$ or $13\pmod{16}$	&										&$\phantom{-}1$\\
					&$3,15\pmod{16}$				&										&$-1$\\ 
			$-2$	
					&3 or $7\pmod{16}$	&$1\pmod{4}$		&$\phantom{-}1$\\
					&11 or $15 \pmod{16}$	&$3\pmod{4}$				&$\phantom{-}1$\\
&otherwise	&				&$-1$\\
			$0\pmod{4}, >0$	&			&$3\pmod{4}$					&$\phantom{-}1$\\
										&			&$1\pmod{4}$					&$-1$\\ 
			2		&$1\pmod{8}$			&3, 5, or $7 \pmod{8}$		&$\phantom{-}1$\\
					&$5\pmod{8}$			&1, 3, or $7 \pmod{8}$		&$\phantom{-}1$\\
					&otherwise				&										&$-1$\\ 
			6		&								&$3\pmod 4$					&$\phantom{-}1$\\
					&								&$1\pmod 4$					&$-1$\\
			$2\pmod{4}, >6$	&			&$7 \pmod{8}$ 				&$\phantom{-}1$\\
					&								&otherwise						&$-1$\\
			\hline
			\end{tabular} 
			\end{table}

		\item For $\nu_2(s)=2\nu_2(t)$, the conditions for the function $w_2^*$ are listed in Table \ref{table: w_2(t) for 2v_2(t)=v(s),v_2(t)odd}.

		\begin{table}[H] 
		\caption{}
		\label{table: w_2(t) for 2v_2(t)=v(s),v_2(t)odd}
		\centering
		\begin{tabular}{l l l l}
		\hline
		$\nu_2(t^2-s)-2\nu_2(t)$	&$t_2$	&$(t^2-s)_2 \pmod{4}$	&$w_2^*(t)$\\
		\hline
		$0,1,3,5\pmod{6}$,$\not=1,3$ &&&$t_2\pmod{4}$\\
		$1$&$1\pmod{8}$&&$1$\\
		&$3\pmod{8}$&&$(t^2-s)_2\pmod{4}$\\
		&$5\pmod{8}$&&$-1$\\
		&$7\pmod{8}$&&$-(t^2-s)_2\pmod{4}$\\
		$2$&$1\pmod{4}$&&$1$\\	
		&$3\pmod{8}$&&$-(t^2-s)_2\pmod{4}$\\
		&$7\pmod{8}$&&$-1$\\
		$2,4\pmod{6}$, $>4$&&&$-(t^2-s)_2\pmod{4}$\\
		$3$&&&$-1$\\
		$4$&$1\pmod{8}$& $5\pmod{8}$&$1$\\
		&$5\pmod{8}$& $1\pmod{8}$&\\
		&$3\pmod{8}$& $1,5,7\pmod{8}$&\\
		&$7\pmod{8}$& $1,3,5\pmod{8}$&\\
		&otherwise&&$-1$\\
		\hline
		\end{tabular} 
		\end{table}
				    \end{itemize}		
					
	\item    For $\nu_2(s)\equiv 3 \pmod{4}$,
			\begin{table}[H]
			\caption{}
			\centering
			\begin{tabular}{l l l r}
			\hline
			 $\nu_2(s)-2\nu_2(t)$	&$s_2$	&$t_2$	&$w_2^*(t)$\\
			\hline
			$\leq -2$	&$1,7\pmod{8}$	&		&$-1$\\
					    &$3,5\pmod{8}$	&		&$\phantom{-}1$\\ 
			$-1$	&$1$ or $3\pmod{8}$	&$1\pmod{4}$			&$-1$\\
				    &	                &$3\pmod{4}$	&$\phantom{-}1$\\
				    &$5$ or $7\pmod{8}$	&$3\pmod{4}$	&$\phantom{-}1$\\
				    &	&$1\pmod{4}$			        &$-1$\\ 
			1		&&$s_2,s_2+2 \pmod{8}$			    &$\phantom{-}1$\\
					&&otherwise							&$-1$\\ 
			$1\pmod{4}, >1$	&			    &$3\pmod 4$	&$\phantom{-}1$\\	
					&								&$1\pmod 4$			&$-1$\\
			3		&$1\pmod{4}$			&3 or $5 \pmod{8}$	&$\phantom{-}1$\\
					&$3\pmod{4}$			&1 or $3 \pmod{8}$ 	&$\phantom{-}1$\\
					&otherwise				&								&$-1$\\
		$3\pmod{4}, >3$	&				&$7\pmod{8}$ 			&$\phantom{-}1$\\
					&								&otherwise				&$-1$\\
					\hline
			\end{tabular}
			\end{table}
	\end{itemize}
\end{prop}

	\begin{proof}
	The proof is obtained from the formulas for $w_2(t)$ in \cite[Proposition 41]{BDD} by reorganizing the formulas by values of $\nu_2(s)\pmod 4$ instead of those of $\nu_2(s)-2\nu_2(t)$, then by computing $w_2^*(t)=w_2(t)\cdot(t^2-s)_2$. \end{proof}

\bibliographystyle{alpha}

\bibliography{bibliography}


\end{document}